\documentclass[12pt]{article}
\usepackage{mathrsfs}
\usepackage{graphicx}
\usepackage{amssymb,amsmath,amsthm, amsfonts}
\usepackage{epsfig}
\usepackage{tikz}
\usepackage{mathrsfs}
\usepackage{setspace}
\usepackage{etoolbox}
\allowdisplaybreaks
\AfterEndEnvironment{proof}{\vspace{2mm}}
\def\C{\mathop{\bf C\kern 0pt}\nolimits}
\def\DD{\mathop{\bf D\kern 0pt}\nolimits}
\def\K{\mathop{\bf K\kern 0pt}\nolimits}
\def\N{\mathop{\bf N\kern 0pt}\nolimits}
\def\Q{\mathop{\bf Q\kern 0pt}\nolimits}
\def\R{\mathop{\bf R\kern 0pt}\nolimits}

\newcommand{\beq}{\begin{equation}}
	\newcommand{\eeq}{\end{equation}}
\newcommand{\ben}{\begin{eqnarray}}
	\newcommand{\een}{\end{eqnarray}}
\newcommand{\beno}{\begin{eqnarray*}}
	\newcommand{\eeno}{\end{eqnarray*}}

\newtheorem{theorem}{Theorem}[section]
\newtheorem{lemma}{Lemma}[section]

\newtheorem{proposition}[theorem]{Proposition}

\theoremstyle{remark}
\newtheorem{remark}{Remark}

\numberwithin{equation}{section}
\allowdisplaybreaks \numberwithin{equation} {section}
\begin{document}
	\title{
		Uniform boundedness and exponential stabilization in a planar degenerate migration-consumption system with logistic source
	}
	\author{Ai Huang$^{1}$, Jing Li$^{2}$, \ Yifu Wang$^{1\dag}$ \ \\[2mm]
		{ \small $^{1}$School of Mathematics and Statistics, Beijing Institute of Technology}\\
		{ \small Beijing 100081, People's Republic of China}\\
		{ \small $^{2}$College of Science, Minzu University of China}\\
		{ \small Beijing 100081, People's Republic of China}}
	\date{}
	\maketitle
	\begin{abstract}
We study the no-flux problem for a planar taxis-type migration-consumption system with logistic source, where
the motility function \(\phi\) satisfies \(\phi(0)=0\)
  and suitably generalizes the prototype \(\phi(s)=s^\alpha\) with \(\alpha\ge1\). For every fixed \(b>0\) and without smallness assumptions, we prove uniform \(L^\infty(\Omega)\times W^{1,\infty}(\Omega)\)-boundedness. For all \(\alpha\ge1\), we further show that \(v\to0\) exponentially in \(W^{1,\infty}(\Omega)\) and \(u\to \frac{a}{b}\) exponentially in \(L^\infty(\Omega)\), with an explicit admissible rate;
  this quantitative stabilization complements the global classical solvability and the \(L^2\)-bound established by Wang and Mu (2025).
  The proof relies on a weighted Gagliardo--Nirenberg inequality and a reciprocal-power functional adapted to
  both the degeneracy and the dissipative structure of the signal equation. These results show that quadratic logistic damping drives the system toward a homogeneous state even though diffusion degenerates at \(v=0\).
	\end{abstract}

	
	\vskip   0.2cm \footnotetext[1]{$^\dag$Corresponding author. E-mail addresses: wangyifu@bit.edu.cn}

	\section{Introduction}
	\setcounter{section}{1}\setcounter{equation}{0}
	
	Chemotaxis, the directed movement of organisms along chemical gradients,
	is a central mechanism in bacterial aggregation and pattern formation (\cite{Colin,Maini}). A classical minimal model is the
	Keller--Segel-type system
	\begin{equation}\label{1.1a}
		\left\{\begin{array}{ll}
			u_{t}=\nabla \cdot (D(v)\nabla u)-\nabla \cdot (uS(v)\nabla v)+f(u,v),\\[1mm]
			v_{t}=\Delta v +g(u,v),
		\end{array}
		\right.
	\end{equation}
where \(u\) denotes the cell density and \(v\)
the concentration of the chemical signal. When $D$ and $S$ are positive constants, $f=0$ and $g(u,v)=u-v$, \eqref{1.1a} reduces to the classical Keller--Segel model
$u_t=\Delta u-\nabla\cdot(u\nabla v)$, $v_t=\Delta v-v+u$, whose dynamical behavior is fairly well understood (see the survey \cite{Bellomo}).
In contrast, when the signal is consumed rather than produced by the cells,
that is, when $f\equiv0$ and $g(u,v)=-uv$, \eqref{1.1a} becomes a chemotaxis-consumption
system. For such systems, global boundedness in two dimensions and eventual
smoothness in three dimensions were established (\cite{Tao Winkler(2012)}),
while in higher dimensions global boundedness of classical solutions was
achieved under a smallness condition on \(\|v_0\|_{L^\infty(\Omega)}\) (\cite{Tao2011});
see also the survey \cite{Lankeit}. In this paper, we focus on the consumption
version of \eqref{1.1a} with \(D(v)=\phi(v)\), \(S(v)=-\phi'(v)\) and a
quadratic logistic source, in the regime where \(\phi(0)=0\) and \(\phi\)
degenerates algebraically near the origin.

In many bacterial populations, motility is not constant but depends on the
local signal concentration and may be suppressed at high signal levels
(\cite{Fu(2012),Liu2011}). In the framework of \eqref{1.1a}, this signal dependence makes the first
equation a Fokker--Planck type equation \(u_t=\Delta(\phi(v)u)\). For
signal-production models, the literature has considered several classes of
motility functions:  when \(\phi\) is uniformly positive and bounded,
global boundedness and stabilization have been obtained for all suitably
regular data in two dimensions (\cite{Tao Winkler(2017)}); when \(\phi\)
decays algebraically or exponentially, critical mass phenomena, boundedness
and the preventive effect of logistic damping have been studied (\cite{Jin2018,Xiao2023,Coon2017,Jin2020,Fujie2021,Xiao2024}).	

In the consumption case, by contrast, the reaction term in the second equation
takes the form \(g(u,v)=-uv\): the nutrient \(v\) is actively degraded by the
cells via contact, driving its concentration toward zero. Experimental
observations reveal a cessation of motility in nutrient-poor environments, a phenomenon linked to colony branching and pattern formation (\cite{Kawasaki,Leyva(2013)}) and also described by individual-level models (\cite{Gude2020,CKL}). This motivates the assumption \(\phi(0)=0\), so that cell diffusion is suppressed near \(v=0\).
In addition, the self-limiting proliferation of the population in such nutrient-poor regimes is commonly described by a logistic term \(au-bu^2\) with \(a,b>0\).
With  these choices, \eqref{1.1a} becomes
	\begin{equation}\label{1.2a}
	u_t=\Delta(\phi(v)u)+au-bu^2,\qquad v_t=\Delta v-uv.
	\end{equation}

When \(\phi\) is positive on \([0,\infty)\), the logistic-free version of
\eqref{1.2a} (i.e., \(u_t=\Delta(\phi(v)u)\), \(v_t=\Delta v-uv\)) possesses
solutions that stabilize toward spatially homogeneous steady states
(\cite{Li Zhao2021,Li Winkler2022,Li Winkler2022R}). The analysis in the
degenerate case \(\phi(0)=0\) becomes substantially more delicate: because
\(\phi(v)\to0\) as \(v\to0\), the \(u\)-equation loses uniform parabolicity,
which makes it more difficult both to derive uniform a priori estimates and to
analyze the large-time behavior. For mildly degenerate \(\phi\) with
\(\phi(s)\simeq s\), global bounded or continuous solutions have been
constructed in one and two dimensions (\cite{Winkler2024}), whereas for strongly
degenerate \(\phi\) the existing global solvability theory is formulated within
generalized solution frameworks
(\cite{Winkler2022Moser,WinklerANS,WinklerZAMP}). Even in these settings, the
large-time behavior is only partially understood: for \(\phi(v)\simeq v^\alpha\) with
\(\alpha\ge1\), only weak-\(*\) stabilization of the transformed density is
available (\cite{Winkler(2023Non)}).

Recent literature has indicated that logistic damping may reduce the
heterogeneity-supporting tendency of cross-diffusion also in doubly degenerate
settings, although the conclusions obtained so far depend on the flux
structure. For a planar doubly degenerate nutrient taxis system with logistic
proliferation and flux \(uv\nabla u-u^2v\nabla v\), global continuous weak
solutions have been constructed and shown to satisfy \(u\rightharpoonup1\) in
\(L^1(\Omega)\) and \(v\to0\) in \(L^\infty(\Omega)\) as $t\to\infty$ (\cite{LiWinkler2026});
for the same flux structure but with stronger degradation \(u-u^\kappa\),
\(\kappa>(n+2)/2\), the global weak solutions satisfy
\(\psi(u(\cdot,t))\overset{*}{\rightharpoonup}0\) in \(L^\infty(\Omega)\) and
\(v(\cdot,t)\to0\) in \(L^p(\Omega)\) for all \(p<\infty\), where
\(\psi(\xi)=(\xi-1)^2/(\xi+1)^2\) (\cite{PanMu2026}).
Note that the flux \(\phi(v)\nabla u+u\phi'(v)\nabla v\) in
\eqref{1.2a} differs from \(uv\nabla u-u^2v\nabla v\) considered in these works. Consequently, the analyses developed there do not carry over directly to \eqref{1.2a}.
For the divergence-form flux in \eqref{1.2a}, the two-dimensional version of
\eqref{1.2a} with \(\alpha\ge1\) and arbitrary \(b>0\) was studied by Wang and
Mu (\cite{Wang2025}), who established global classical solvability and a
uniform \(L^2(\Omega)\) bound for \(u\); while in higher dimensions,
 they further showed that, for \(\alpha>1\), the global weak solutions become eventually smooth
(\cite[Theorem~1.5]{Wang2025}). These results, however, do not address the uniform
$L^\infty(\Omega)\times W^{1,\infty}(\Omega)$-boundedness
of solutions or their convergence behavior.

  	\textbf{Main results.} Motivated by the above, we study the initial-boundary value problem
	\begin{equation}\label{1.1}
		\left\{
		\begin{array}{ll}
			u_t=\Delta(\phi(v)u)+au-bu^2, & x\in\Omega,\ t>0,\\[2mm]
			v_t=\Delta v-vu, & x\in\Omega,\ t>0,\\[2mm]
			\partial_\nu u=\partial_\nu v=0, & x\in\partial\Omega,\ t>0,\\[2mm]
			u(\cdot,0)=u_0,\quad v(\cdot,0)=v_0, & x\in\Omega,
		\end{array}
		\right.
	\end{equation}
	where \(\Omega\subset\mathbb R^2\) is a bounded domain with smooth
	boundary.
	Here \(a>0\), \(b>0\), and the initial data \((u_0,v_0)\)
	satisfy
	\begin{equation}\label{1.4}
		\left\{\begin{array}{ll}
			u_0\in W^{1,\infty}(\Omega)\mbox{ is nonnegative with }u_0\not\equiv0,\\[1mm]
			v_0\in W^{1,\infty}(\Omega)\mbox{ with }v_0>0\mbox{ in }\overline{\Omega}.
		\end{array}
		\right.
	\end{equation}
The motility function \(\phi\) is assumed to suitably generalize the prototype given by \(\phi(s):=s^\alpha\) ($\alpha\ge1$) near zero, in the sense that
		\begin{equation}\label{1.5}
		\phi\in C^0([0,\infty))\cap C^3((0,\infty))\mbox{ is such that }\phi(0)=0\mbox{ and }\phi>0\mbox{ on }(0,\infty),
	\end{equation}
	and that for some $\alpha\geq 1$,
	\begin{equation}\label{1.6}
		\liminf_{s\searrow0}\frac{\phi(s)}{s^\alpha}>0,\qquad
		\limsup_{s\searrow0}\frac{|\phi'(s)|}{s^{\alpha-1}}<\infty.
	\end{equation}
	The existence of a  classical solution to \eqref{1.1} was established in \cite{Wang2025}; see Proposition \ref{pro1.1} below.

More importantly, this is not merely a gap in \cite{Wang2025}:
the $L^2(\Omega)$ bound does not automatically upgrade to
$L^\infty(\Omega)\times W^{1,\infty}(\Omega)$. As noted above, the
degeneracy $\phi(0)=0$ destroys uniform parabolicity, so the standard
Moser absorption of $\phi'(v)u^{p-1}\nabla v\cdot\nabla u$ fails, and
a new iteration closure is required. This is the main point of
Theorem~\ref{theorem1.1}.
	
	
	\begin{theorem}\label{theorem1.1}
		Under assumptions \eqref{1.4}--\eqref{1.6}, any global classical solution $(u,v)$ of problem \eqref{1.1} satisfies
		\begin{equation}
			\|u(\cdot,t)\|_{L^{\infty}(\Omega)}+\|v(\cdot,t)\|_{W^{1,\infty}(\Omega)}\leq C~~\mbox{for all}~~t>0,
		\end{equation}
		where $C>0$ depends only on the initial data and parameters.
	\end{theorem}
The proof of Theorem~\ref{theorem1.1} is based on a bootstrap argument
that circumvents the loss of uniform parabolicity mentioned above. The starting point is the weighted gradient dissipation
of Lemma~\ref{lemma3.3}, which supplies a control of
$\int_\Omega u|\nabla v|^q/v^{q-1}$ that replaces the missing diffusion
dissipation of $u$. Combined with a weighted Gagliardo--Nirenberg
inequality (Lemma~\ref{le3.2}), this controls the mixed
integrals $\int_\Omega u^\beta v^\alpha$ and yields uniform
$L^p(\Omega)$ bounds for $u$ for all $p>2$, starting from the $L^2$
bound of Lemma~\ref{L2bound}. Since a standard Moser iteration based on
uniform parabolicity is unavailable, the final $L^\infty$ closure uses
the pointwise gradient bound $|\nabla v|\le Cv^{1-r}$
(Lemma~\ref{lem3.5}) together with the two-dimensional
Gagliardo--Nirenberg inequality.

In contrast to the logistic-free scenario, in which nutrient depletion can
preserve spatial heterogeneity for suitable initial data
(\cite[Corollary~1.4]{Winkler(2023Non)}), quadratic logistic damping is shown
below to drive the system exponentially to the homogeneous state \((\frac{a}{b},0)\) for all $\alpha\geq 1$.
		\begin{theorem}\label{theorem1.2}
		Assume that \eqref{1.4}--\eqref{1.6} hold with \(\alpha\ge1\).
		Then there exists a constant \(C>0\) such that
		\begin{equation}\label{1.10}
			\|v(\cdot,t)\|_{W^{1,\infty}(\Omega)}
			\leq C e^{-\frac{a}{2b} t}
			\qquad\text{for all }t>0.
		\end{equation}
		Moreover, for any \(\sigma\in(0,\min\{a,\frac{\alpha a }{8b}\})\),
		there exists a constant \(C_\sigma>0\), depending on $\sigma$, such that
		\begin{equation}\label{1.11a}
			\left\|u(\cdot,t)-\frac{a}{b}\right\|_{L^\infty(\Omega)}
			\leq C_\sigma e^{-\sigma t}
			\qquad\text{for all }t>0.
		\end{equation}
	\end{theorem}
The proof of Theorem~\ref{theorem1.2} is divided into two steps.
\noindent\text{Step 1.}
Using the uniform bound for $u$, we first prove that $v\to0$ uniformly and derive
the relative gradient estimate $|\nabla v|\le Cv$ for $t\ge1$
(Lemma~\ref{lem4.1}). This estimate controls the cross-diffusion remainder in the subsequent argument.
The reciprocal-power functional $\Psi(s)=s^{-4}+4s-5$ then yields
$u(\cdot,t)\to \frac{a}{b}$ in $L^2(\Omega)$, without requiring a time-independent
positive lower bound for $u$ (Lemma~\ref{lemma4.3}).
The signal equation consequently gives exponential decay of $v$ in
$W^{1,\infty}(\Omega)$ (Lemma~\ref{lemma4.4}).
\noindent\text{Step 2.}
The logistic term contributes the factor-free negative term $-2aE(t)$ to the estimate for
$E'(t)$, where functional $E(t)=\int_\Omega\Psi(u(\cdot,t)/(\frac{a}{b}))$.
Together with the exponential decay of the cross-diffusion coefficient, this
yields exponential decay of $E(t)$ (Lemma~\ref{lem4.5}).
To obtain uniform convergence, we apply a Moser-type iteration to the positive
parts of the rescaled deviations of $u/(\frac{a}{b})$ and its reciprocal
(Lemma~\ref{lem4.6}). The relative gradient estimate ensures that the diffusion
and cross-diffusion terms contain the same decaying factor.
Retaining this factor in the interpolation estimate allows us to close the
iteration and obtain exponential convergence in $L^\infty(\Omega)$.
	
\begin{remark}\label{rem:kappa}
	Whether the quadratic damping is merely convenient for the present
	approach or genuinely critical for the dynamics---i.e.\ whether
	\(au-bu^\kappa\) with \(\kappa<2\) still yields global boundedness
	and stabilization---is left open.
	\end{remark}

	The rest of this paper is organized as follows. Section 2 collects preliminary estimates and functional inequalities. In Section 3, we establish the uniform $L^\infty$-boundedness of solutions via a Moser-type iteration. Section 4
is devoted to the large-time behavior and exponential convergence
	through heat semigroup estimates, a reciprocal-power functional and a Moser-type iteration.
	\vskip3mm
	
	\section{Preliminaries}
	In this section, we collect some preliminary results that will be used throughout the paper.
	In the sequel, unless otherwise stated, $(u,v)$ denotes a classical solution to \eqref{1.1}, the existence of which is warranted by the following result.
	
	\begin{proposition}(\cite[Theorem~1.2]{Wang2025})\label{pro1.1}
		Let $\Omega\subset\mathbb{R}^2$ be a bounded domain with smooth boundary. Assume that $\alpha\geq 1$, $\phi$
		satisfies \eqref{1.5}--\eqref{1.6}, and $(u_0,v_0)$ fulfills \eqref{1.4}. Then for $a>0$ and $b>0$,  problem \eqref{1.1} admits
		a  classical solution  $(u,v)\in (C^0(\overline\Omega\times [0,\infty))\cap  C^{2,1}(\overline\Omega \times (0,\infty)))^2$, where $u(x,t)>0$ and $v(x,t)>0$ for all $x\in \overline \Omega$ and $t>0$.
	\end{proposition}
	
	To exploit the singularly weighted dissipation rates arising from the evolution of $v$, we shall frequently make use of the following inequalities.
	
	\begin{lemma}(\cite{WinklerDCDSB})\label{le2.3}
		Let $ q\geq2$, and $\varphi\in C^2(\overline{\Omega})$ be positive  satisfying $\frac{\partial\varphi}{\partial\nu}=0$ on $\partial\Omega$. Then
		\begin{equation*}
			\begin{split}			\int_{\Omega}\varphi^{-q-1}|\nabla\varphi|^{q+2}\leq(q+\sqrt{2})^2
				\int_{\Omega}\varphi^{-q+3}|\nabla\varphi|^{q-2}|D^2\ln\varphi|^2
			\end{split}
		\end{equation*}
		and
		\begin{equation*}
			\begin{split}
				\int_{\Omega}\varphi^{-q+1}|\nabla\varphi|^{q-2}|D^2\varphi|^2\leq(q+\sqrt{2}+1)^2\int_{\Omega}\varphi^{-q+3}|\nabla\varphi|^{q-2}|D^2\ln\varphi|^2.
			\end{split}
		\end{equation*}
	\end{lemma}
	We begin with some basic properties of $(u,v)$.
	\begin{lemma}\label{le2.4}
		Let $\Omega\subset\mathbb R^2$ be a bounded domain with smooth boundary, and let $a,b>0$. Under the
		assumptions \eqref{1.4}--\eqref{1.6},  the following estimates hold:
		\begin{equation}\label{2.1}
			\int_{\Omega}u(\cdot,t)\leq m:=\max\{\int_{\Omega} u_0,\frac{a|\Omega|}{b}\}~~~~\mbox{for all}~t\in(0,\infty).
		\end{equation}
		\begin{equation}\label{2.5*}
			\int_t^{t+1}\int_{\Omega}u^2\leq\frac{(a+1)m}{b}~~~~\mbox{for all}~t\in(0,\infty).
		\end{equation}
		as well as
		\begin{equation} \label{2.2}
			v(x,t)\leq\|v_{0}\|_{L^{\infty}(\Omega)}~~~~\mbox{for all}~(x,t)\in\bar{\Omega}\times [0,\infty).
		\end{equation}
	\end{lemma}
	\begin{proof}[Proof]
		By the H\"{o}lder inequality, integrating the first equation in \eqref{1.1} yields
		\begin{equation*}
			\frac{d}{dt}\int_{\Omega}u= a\int_{\Omega}u-b\int_{\Omega}u^2\leq a\int_{\Omega}u-\frac{b}{|\Omega|}\bigg(\int_{\Omega}u\bigg)^2 \quad\hbox{for all}~t>0,
		\end{equation*}
		which along with an ODE comparison argument implies \eqref{2.1}.
		Then \eqref{2.5*} follows by integrating the same identity and applying the Cauchy--Schwarz inequality. Finally, \eqref{2.2} is a direct consequence of the maximum principle applied to the second equation in \eqref{1.1}, since $uv\ge0$.
	\end{proof}

	Based on the classical $L^p-L^q$ estimates for the Neumann heat semigroup $(e^{t\Delta})_{t\geq0}$ in \cite{Winkler(2010)},
	the regularity of $v$  can be enhanced through the higher integrability of $u$.
	
	\begin{lemma}\label{le2.5} Let $\Omega \subset \mathbb R^2$ be  a smoothly bounded domain. Then for all $p\geq1$ and
		\begin{equation*}
			\begin{aligned}
				\theta\in\left\{
				\begin{array}{l}
					{\left({1,\frac{2p}{2-p}}\right)},\quad p< 2, \\
					{(1,\infty)},\quad\quad p=2,\\
					{(1,\infty]},\quad\quad p>2,
				\end{array}\right.
			\end{aligned}
		\end{equation*}
		there exists a constant $C>0$ such that
		\begin{equation}
			\|v(\cdot,t)\|_{W^{1,\theta}(\Omega)}\leq C(1+\sup\limits_{s>0}\|u(\cdot,s)\|_{L^p(\Omega)})~~\mbox{for all} ~t\in(0,\infty).
		\end{equation}
	\end{lemma}

\begin{lemma}\label{lem:heat_kernel_bound}
	Let $\Omega\subset\mathbb{R}^2$ be a smoothly bounded domain. For any fixed $\tau>0$, there exist constants \(C_1(\tau), C_2(\tau)>0\) such that for any nonnegative function \(f\in L^1(\Omega)\), the Neumann heat semigroup \(e^{\tau\Delta}\) satisfies
	\begin{equation}\label{eq:heat_semigroup_bound}
		C_1(\tau)\|f\|_{L^1(\Omega)}
		\le (e^{\tau\Delta} f)(x)
		\le C_2(\tau)\|f\|_{L^1(\Omega)}
		\quad\mbox{for all }x\in\overline\Omega.
	\end{equation}
\end{lemma}
\begin{proof}
	Let \(K(\tau,x,y)\) denote the Neumann heat kernel in $\Omega$. Since $\Omega$ is smooth and bounded, \(K(\tau,\cdot,\cdot)\) is continuous and strictly positive on \(\overline\Omega\times\overline\Omega\) for any fixed \(\tau>0\) (see \cite[proof of Lemma~3.1]{WinklerYokota2018}). Consequently, there exist constants \(C_1(\tau), C_2(\tau)>0\) such that
	\[
	C_1(\tau)\le K(\tau,x,y)\le C_2(\tau)
	\quad\mbox{for all }x,y\in\overline\Omega.
	\]
	Using the representation \((e^{\tau\Delta} f)(x)=\int_\Omega K(\tau,x,y)f(y)\,dy\) and the fact that $f\ge0$, we immediately obtain \eqref{eq:heat_semigroup_bound}.
\end{proof}

	\section{Uniform \(L^\infty\)-bounds }
	
	This section establishes higher-order integrability estimates for $u$ through a bootstrap argument. A key ingredient is the following lemma on the dissipative structure of the weighted functional $\int_\Omega v^{-q+1}|\nabla v|^q$ for $q\ge2$.
	The dissipative property derived below, especially the third term on the left-hand side of \eqref{3.1**}, will
	be essential in compensating the cross-diffusive effects that arise in the $L^p$-testing procedure for $u$.
	
	\begin{lemma}\label{lemma3.3}
		Let $\Omega\subset\mathbb{R}^2$ be a smoothly bounded domain and $q\geq 2$. Then there exists $C(q)>0$ such that for all $t\in(0,\infty)$,
		\begin{equation}\label{3.1**}
			\begin{split}
				&\frac{d}{dt}\int_{\Omega}\frac{|\nabla v|^q}{v^{q-1}}+\frac{q}{8(q+\sqrt{2})^2}\int_{\Omega}\frac{|\nabla v|^{q+2}}{v^{q+1}}+(q-1)^2\int_\Omega u\frac{|\nabla v|^q}{v^{q-1}}
				\leq  C(q)\bigg\{\int_{\Omega}u^{\frac{q+2}{2}}v+\int_\Omega v\bigg\}.
			\end{split}
		\end{equation}
	\end{lemma}
	\begin{proof}
		According to the second equation in \eqref{1.1} as well as the identities $\nabla v\cdot\nabla\Delta v=\frac{1}{2}\Delta|\nabla v|^2-|D^2v|^2$ and $\nabla|\nabla v|^2=2D^2v\cdot\nabla v$, several integrations by parts show that
		\begin{align}\label{3.9}
			& \frac{d}{dt}\int_\Omega \frac{|\nabla v|^q}{v^{q-1}}\nonumber\\
			=&-\frac{q(q-2)}{4}\int_\Omega v^{-q+1}|\nabla v|^{q-4}|\nabla|\nabla v|^2|^2+q(q-1)\int_\Omega v^{-q}|\nabla v|^{q-2}\nabla v\cdot\nabla|\nabla v|^2\nonumber\\
			&-q\int_\Omega v^{-q+1}|\nabla v|^{q-2}|D^2 v|^2-q(q-1)\int_\Omega \frac{|\nabla v|^{q+2}}{v^{q+1}}\\
			&+\frac{q}{2}\int_{\partial\Omega}v^{-q+1}|\nabla v|^{q-2}\cdot\frac{\partial|\nabla v|^2}{\partial\nu}+q(q-2)\int_\Omega u v^{-q+2}|\nabla v|^{q-4}\nabla v\cdot(D^2 v\cdot\nabla v)\nonumber\\
			&+q\int_\Omega u v^{-q+2}|\nabla v|^{q-2}\Delta v-(q-1)^2\int_\Omega u \frac{|\nabla v|^q}{v^{q-1}}\nonumber
		\end{align}
		for all $t\in(0,\infty)$. Due to the fact that for all $\varphi\in C^2(\Omega)$ with $\varphi>0$,
		\begin{equation*}
			|D^2\ln\varphi|^2=\frac{1}{\varphi^2}|D^2\varphi|^2-\frac{2}{\varphi^3}
			\nabla\varphi\cdot(D^2\varphi\cdot\nabla\varphi)
			+\frac{1}{\varphi^4}|\nabla\varphi|^4,
		\end{equation*}
		we get
		\begin{equation*}
			|D^2v|^2=v^2|D^2\ln v|^2+\frac{1}{v}\nabla v\cdot\nabla|\nabla v|^2-\frac{1}{v^2}|\nabla v|^4,
		\end{equation*}
		and then rewrite
		\begin{align}\label{3.10}
			&-\frac{q(q-2)}{4}\int_\Omega v^{-q+1}|\nabla v|^{q-4}|\nabla|\nabla v|^2|^2+q(q-1)\int_\Omega v^{-q}|\nabla v|^{q-2}\nabla v\cdot\nabla|\nabla v|^2\nonumber\\
			&-q\int_\Omega v^{-q+1}|\nabla v|^{q-2}|D^2 v|^2-q(q-1)\int_\Omega \frac{|\nabla v|^{q+2}}{v^{q+1}}\nonumber\\
			=&-q\int_\Omega v^{-q+3}|\nabla v|^{q-2}|D^2\ln v|^2-\frac{q(q-2)}{4}\int_\Omega v^{-q+1}|\nabla v|^{q-4}|\nabla|\nabla v|^2|^2\\
			&+q(q-2)\int_\Omega v^{-q}|\nabla v|^{q-2}\nabla v\cdot\nabla|\nabla v|^2-q(q-2)\int_\Omega \frac{|\nabla v|^{q+2}}{v^{q+1}}\nonumber\\
			=&-q\int_\Omega v^{-q+3}|\nabla v|^{q-2}|D^2\ln v|^2-\frac{q(q-2)}{4}\int_\Omega v^{-q+1}|\nabla v|^{q-4}\bigg|\nabla|\nabla v|^2-\frac{2}{v}|\nabla v|^2\nabla v\bigg|^2,\nonumber
		\end{align}
		while using $|\Delta v|\leq\sqrt{2}|D^2v|$, we obtain that
		\begin{equation}\label{2.18}
			\begin{split}
				&q(q-2)\int_\Omega u v^{-q+2}|\nabla v|^{q-4}\nabla v\cdot(D^2 v\cdot\nabla v)+q\int_\Omega u v^{-q+2}|\nabla v|^{q-2}\Delta v\\
				\leq& q(q-2)\int_\Omega u v^{-q+2}|\nabla v|^{q-2}|D^2v|+q\sqrt{2}\int_\Omega u v^{-q+2}|\nabla v|^{q-2}|D^2v|
			\end{split}
		\end{equation}
		for all $t\in(0,\infty)$. According to \cite[Lemma 3.5]{WinklerDCDSB}, there exists $c_1(q)>0$ fulfilling
		\begin{equation}\label{2.19}
			\begin{split}
				&\frac{q}{2}\int_{\partial\Omega}v^{-q+1}|\nabla v|^{q-2}\cdot\frac{\partial|\nabla v|^2}{\partial\nu}\\
				\leq &\frac{q}{4(q+\sqrt{2}+1)^2}\int_\Omega v^{-q+1}|\nabla v|^{q-2}|D^2 v|^2+\frac{q}{4(q+\sqrt{2})^2}\int_\Omega \frac{|\nabla v|^{q+2}}{v^{q+1}}+c_1(q)\int_\Omega v.
			\end{split}
		\end{equation}
		Furthermore, by Lemma \ref{le2.3}, we have
		\begin{align}\label{3.13}
			&~~q\int_\Omega v^{-q+3}|\nabla v|^{q-2}|D^2\ln v|^2\\
			&\geq \frac{q}{2(q+\sqrt{2}+1)^2}\int_\Omega v^{-q+1}|\nabla v|^{q-2}|D^2 v|^2+\frac{q}{2(q+\sqrt{2})^2}\int_\Omega \frac{|\nabla v|^{q+2}}{v^{q+1}}.\nonumber
		\end{align}
		Substituting \eqref{3.10}--\eqref{3.13} into \eqref{3.9}, we get
		\begin{align}\label{3.14}
			&\frac{d}{dt}\int_\Omega\frac{|\nabla v|^q}{v^{q-1}}+\frac{q}{4(q+\sqrt{2}+1)^2}\int_\Omega v^{-q+1}|\nabla v|^{q-2}|D^2 v|^2+\frac{q}{4(q+\sqrt{2})^2} \int_\Omega \frac{|\nabla v|^{q+2}}{v^{q+1}}\nonumber\\
			&\quad+(q-1)^2\int_\Omega u\frac{|\nabla v|^q} {v^{q-1}}\\
			\leq & q(q+\sqrt{2}-2)\int_\Omega u v^{-q+2}|\nabla v|^{q-2}|D^2 v|+c_1(q)\int_\Omega v.\nonumber
		\end{align}
		Now employing Young's inequality we can estimate
		\begin{align*}
			&q(q+\sqrt{2}-2)\int_\Omega u v^{-q+2}|\nabla v|^{q-2}|D^2v|\\
			\leq &\frac{q}{4(q+\sqrt{2}+1)^2}\int_\Omega v^{-q+1}|\nabla v|^{q-2}|D^2v|^2+\frac{q^2(q+\sqrt{2}+1)^2(q+\sqrt{2}-2)^2}{q}\int_\Omega u^2v^{-q+3}|\nabla v|^{q-2}\\
			\leq &\frac{q}{4(q+\sqrt{2}+1)^2}\int_\Omega v^{-q+1}|\nabla v|^{q-2}|D^2v|^2\\
			&\quad+\int_\Omega\bigg\{\frac{q}{8(q+\sqrt{2})^2}\frac{|\nabla v|^{q+2}}{v^{q+1}}\bigg\}^{\frac{q-2}{q+2}}\cdot
			\bigg\{\bigg(\frac{8(q+\sqrt{2})^2}{q}\bigg)^{\frac{q-2}{q+2}}
			\cdot256q^5u^2v^{\frac{4}{q+2}}\bigg\}\\
			\leq &\frac{q}{4(q+\sqrt{2}+1)^2}\int_\Omega v^{-q+1}|\nabla v|^{q-2}|D^2v|^2+\frac{q}{8(q+\sqrt{2})^2}\int_\Omega\frac{|\nabla v|^{q+2}}{v^{q+1}} \\ &\quad+\bigg\{\frac{8(q+\sqrt{2})^2}{q}\bigg\}^{\frac{q-2}{4}}
			\cdot\bigg\{256q^5\bigg\}^{\frac{q+2}{4}}\int_{\Omega}u^{\frac{q+2}{2}}v,
		\end{align*}
		which along with \eqref{3.14}  yields \eqref{3.1**} with the desired $C(q)$.
	\end{proof}
	
	On the basis of  inequality \eqref{3.1**},
	we next establish the uniform $L^2$-boundedness of $u$, which is a starting point for the subsequent bootstrap argument.
	
	
	
	
	\begin{lemma}\label{L2bound}
		Let $\Omega\subset\mathbb{R}^2$ be a smoothly bounded domain, $a>0$ and $b>0$.
		Suppose that \eqref{1.5}--\eqref{1.6} hold with $\alpha\ge1$, and that \eqref{1.4} is valid.
		Then there exists a constant $C>0$ such that
		\begin{equation}\label{3.8}
			\|u(\cdot,t)\|_{L^2(\Omega)} \le C \quad \text{for all } t>0.
		\end{equation}
	\end{lemma}
	\begin{proof}Applying Lemma \ref{lemma3.3} to $q:=2$, we can find some $c_1>0$ and $c_2>0$ such that
		\begin{align}\label{3.9*}
			\frac{d}{dt}\int_\Omega \frac{|\nabla v|^2}{v}+\frac{1}{128}\int_\Omega \frac{|\nabla v|^{4}}{v^3}+\int_\Omega u\frac{|\nabla v|^2}{v}\leq c_1\bigg\{\int_{\Omega}u^{2}+\int_\Omega v\bigg\}
		\end{align}
		and
		\begin{equation}\label{3.11**}
			\frac{d}{dt}\int_\Omega \frac{|\nabla v|^2}{v}+\int_\Omega\frac{|\nabla v|^2}{v}\leq c_2\bigg\{\int_{\Omega}u^{2}+\int_\Omega v\bigg\},
		\end{equation}
		where the latter follows from the elementary inequality
		$\int_\Omega\frac{|\nabla v|^2}{v}\leq\frac{1}{128}\int_{\Omega}\frac{|\nabla v|^4}{v^3}+32\int_{\Omega}v.$
		
		In view of \eqref{2.5*},  \eqref{3.11**} implies that $\int_{\Omega}\frac{|\nabla v|^2}{v}(\cdot,t)\leq c_3$ for some $c_3>0$  and all $t>0$.
		Integrating \eqref{3.9*} over the interval $(t,t+1)$ then gives
		$$\int_t^{t+1}\int_{\Omega}\frac{|\nabla v|^4}{v^3}
		\leq c_4~~\mbox{for all}~~t>0.$$
		
		Next, testing the first equation in
		\eqref{1.1} by $u$ and applying Young's inequality, together with  \eqref{1.6}, we can find $c_5>0$ and $c_6>0$ such that
		\begin{align}\label{3.11}
			\frac{d}{dt}\int_{\Omega}u^2+ c_5\int_{\Omega} v^\alpha|\nabla u|^2+2b\int_{\Omega} u^3\leq c_6\int_{\Omega} u^2v^{\alpha-2}|\nabla v|^2+2a\int_{\Omega} u^2.
		\end{align}
		
		By Young's inequality,  the first term on the right-hand side of \eqref{3.11}   can be estimated as
		\begin{equation}\label{3.12}
			c_6\int_{\Omega} u^2v^{\alpha-2}|\nabla v|^2\leq c_5\int_\Omega v^\alpha|\nabla u|^2+c_7\bigg\{\int_{\Omega}u^2+1\bigg\}\bigg\{\int_{\Omega}\frac{|\nabla v|^4}{v^3}+1\bigg\}
		\end{equation}
		with some $c_7>0$; for a detailed derivation of this estimate, we refer to \cite[Lemma 3.7]{Wang2025}.
		
		Therefore, inserting \eqref{3.12} into \eqref{3.11} leads to
		\begin{align*}
			\frac{d}{dt}\int_{\Omega}u^2+2b\int_{\Omega} u^3
			\leq (2a+c_7) \bigg\{\int_{\Omega}u^2+1\bigg\}\bigg\{\int_{\Omega}\frac{|\nabla v|^4}{v^3}+1\bigg\}.
		\end{align*}
		Since, by Young's inequality,
		\begin{align*}
			b|\Omega|^{-\frac{1}{2}}\bigg\{\int_{\Omega}u^2 +1\bigg\}^{\frac{3}{2}}\leq \sqrt{2}b|\Omega|^{-\frac{1}{2}}\bigg\{\int_{\Omega}u^2\bigg\}^{\frac{3}{2}}
			+\sqrt{2}b|\Omega|^{-\frac{1}{2}}\leq2b\int_{\Omega}u^3+2b|\Omega|^{-\frac{1}{2}},
		\end{align*}
		it follows that
		\begin{align}\label{3.13**}
			\frac{d}{dt}\int_{\Omega}u^2+	b|\Omega|^{-\frac{1}{2}}\bigg\{\int_{\Omega}u^2 +1\bigg\}^{\frac{3}{2}}\leq c_{8}\bigg\{\int_{\Omega}u^2 +1\bigg\}
			\bigg\{\int_{\Omega}\frac{|\nabla v|^4}{v^3}+1\bigg\}
		\end{align}
		where $c_8=2a+c_7+ 2b|\Omega|^{-\frac{1}{2}}$ for all $t>0$.
		Writing $y(t):=\int_{\Omega}u^2+1$ and $g(t):=c_{8}\int_{\Omega}\frac{|\nabla v|^4}{v^3}+c_{8}$, it follows from \eqref{3.13**} that
		\begin{align*}
			y'(t)+b|\Omega|^{-\frac{1}{2}}y^{\frac{3}{2}}(t)\leq y(t)g(t)
		\end{align*}
		for all $t>0$, which, by a standard ODE comparison argument, yields \eqref{3.8} with the desired constant $C>0$.
	\end{proof}	
	
	The following lemma provides a weighted interpolation inequality that will serve as the key tool for handling the cross-diffusive term in the $L^p$-testing procedure. It enables us to dominate mixed integrals of the form
	$\int_\Omega u^\beta v^\alpha$ by the available $u$-related dissipation and the weighted gradient dissipation of $v$ from Lemma \ref{lemma3.3}.
	The inequality is derived in the two-dimensional setting and does not require convexity of
	$\Omega$.
	
		\begin{lemma}\label{le3.1}
		Let $\Omega \subset \mathbb R^2$ be  a smoothly bounded domain, and suppose that $p>2$, $\alpha>0$ and $1< r<\infty$. Then for any $\eta>0$ one can find a constant $C(\eta,p,r,\alpha,\Omega)>0$
		such that
		\begin{align}\label{3.1}
			\|\varphi^{\frac{p}{2}}\psi^{\frac{\alpha}{2}}\|^2_{L^{r}(\Omega)}\leq \eta \int_{\Omega}\varphi^{p-2}\psi^\alpha|\nabla \varphi|^2+\eta
			\int_{\Omega}\varphi^{p}\psi^{\alpha-2}|\nabla \psi|^2+C(\eta, p,r,\alpha,\Omega)\cdot
			\left\{\int_{\Omega}\varphi\right\}^{p}\left\{\sup\limits_{x\in\Omega}\psi(x)\right\}^\alpha
		\end{align}for all $\varphi, \psi\in C^1(\overline{\Omega})$  with $\varphi,\psi>0$ in $\overline{\Omega}$.
	\end{lemma}
	\begin{proof} From the version of the Gagliardo--Nirenberg inequality (\cite{Winkler(2010)}),   it follows that for $1<r<\infty$ there exists $c_1=c_1(r,\Omega)>0$ such that
		\begin{align}\label{3.1*}
			\|\varphi^{\frac{p}{2}}\psi^{\frac{\alpha}{2}}\|^2_{L^{r}(\Omega)}&\leq c_1\|\nabla
			(\varphi^{\frac{p}{2}}\psi^{\frac{\alpha}{2}})\|^{2\theta}_{L^{2}(\Omega)}
			\|\varphi^{\frac{p}{2}}\psi^{\frac{\alpha}{2}}\|^{2(1-\theta)}_{L^{\frac{2}{p}}(\Omega)}
			+c_1\|\varphi^{\frac{p}{2}}\psi^{\frac{\alpha}{2}}\|_{L^{\frac{2}{p}}(\Omega)}^2
		\end{align}
		with $\theta:=1-\frac{2}{rp}$.
		For any $\eta_1>0$,
		the Young inequality leads to
		\begin{align*}
			c_1\|\nabla
			(\varphi^{\frac{p}{2}}\psi^{\frac{\alpha}{2}})\|^{2\theta}_{L^{2}(\Omega)}
			\|\varphi^{\frac{p}{2}}\psi^{\frac{\alpha}{2}}\|^{2(1-\theta)}_{L^{\frac{2}{p}}(\Omega)}
			&\leq \eta_1 \|\nabla (\varphi^{\frac{p}{2}}\psi^{\frac{\alpha}{2}})\|_{L^{2}(\Omega)}^2
			+c_2(\eta_1,\theta)\|\varphi^{\frac{p}{2}}\psi^{\frac{\alpha}{2}}\|_{L^{\frac{2}{p}}(\Omega)}^2,
		\end{align*}
		where $c_2(\eta_1,\theta)={\eta_1}^{\frac{\theta}{\theta-1}} {c_1}^{\frac{1}{1-\theta}}.
		$
		At this position,  an application  of the H\"{o}lder inequality yields
		\begin{equation}\label{3.2d}
			\begin{split}
				\|\varphi^{\frac{p}{2}}\psi^{\frac{\alpha}{2}}\|_{L^{\frac2{p}}(\Omega)}^2
				&=\bigg\{\int_\Omega\varphi\psi^{\frac{\alpha}{p}}\bigg\}^{p}	\leq\bigg\{\sup\limits_{x\in\Omega}|\psi(x)|\bigg\}^{\alpha}\cdot\bigg\{\int_\Omega\varphi\bigg\}^{p}.
			\end{split}
		\end{equation}
		In addition, a simple calculation shows that
		\begin{align}\label{3.2e}
			\eta_1\|\nabla
			(\varphi^{\frac{p}{2}}\psi^{\frac{\alpha}{2}})\|_{L^{2}(\Omega)}^2
			&=\eta_1\int_\Omega\bigg|\frac{p}{2}\varphi^{\frac{p-2}{2}}\psi^{\frac{\alpha} 2}\nabla\varphi+\frac{\alpha}2 \varphi^{\frac{p}{2}}\psi^{\frac{\alpha-2} 2}\nabla\psi\bigg|^2\nonumber\\
			&\leq \eta_1p^2\int_\Omega\varphi^{p-2}\psi^{\alpha}|\nabla\varphi|^2
			+\eta_1\alpha^2 \int_\Omega\varphi^{p}\psi^{\alpha-2}|\nabla\psi|^2\nonumber\\
			&\leq\eta\int_\Omega\varphi^{p-2}\psi^{\alpha}|\nabla\varphi|^2
			+\eta\int_\Omega\varphi^{p}\psi^{\alpha-2}|\nabla\psi|^2
		\end{align}
		by the choice of	$
		\eta_1:=\min\{\frac{\eta}{p^2},\frac{\eta}{\alpha^2}\}$. From \eqref{3.1*}--\eqref{3.2e}, it follows that \eqref{3.1} holds with the desired $C(\eta,p,r,\alpha,\Omega)>0$. 
	\end{proof}
	
	Lemma \ref{le3.1} can be further refined to yield explicit control of $\int_\Omega \varphi^\beta\psi^\alpha$ in terms of the dissipative quantities $\int_\Omega \varphi^{p-2}\psi^\alpha|\nabla\varphi|^2$ and the singular weighted integral $\int_\Omega \varphi\psi^{-q+1}|\nabla\psi|^q$, provided that the parameters $\alpha,\beta,p,q$ satisfy certain compatibility conditions. This refinement is recorded in the following lemma.
	
	\begin{lemma}\label{le3.2}
		Let $\Omega\subset\mathbb{R}^2$ be a smoothly bounded domain, $L>0$, $p_0\geq1$, $p>2$, $q>2+\frac{2p}{p_0}$, $\alpha\geq1$ and $\beta\in [1,p+p_0)$. Then for any $\eta\in(0,1)$, there exists $C=C(\eta,L,p,\beta,p_0,q,\alpha,\Omega
		)>0$ such that for all  $\varphi\in C^1(\overline \Omega)$ and $\psi \in C^1(\overline \Omega)$ fulfilling $\varphi> 0$, $\psi>0$ in $\overline \Omega$ and $\int_{\Omega}\varphi^{p_0}\leq L$, we have
		\begin{align*}
			\int_{\Omega}\varphi^{\beta} \psi^{\alpha}&\leq \eta \int_{\Omega}\varphi^{p-2}\psi^{\alpha}|\nabla
			\varphi|^2+\eta \int_{\Omega}\varphi{\psi}^{-q+1}|\nabla  \psi|^q+C\left\{\sup\limits_{x\in\Omega}\psi(x)+1\right\}^\tau\left\{\int_{\Omega}\varphi+\left(\int_{\Omega}\varphi\right)^{p}\right\}
		\end{align*}
		with $\tau:=\frac{pq(\alpha-1)}{q-2}+1$.
	\end{lemma}
	\begin{proof}It is observed that the restriction $q>2+\frac{2p}{p_0}$ ensures that $\frac{pq}{q-2}<p+p_0$.
		Firstly, we verify this lemma in the case
		\begin{equation}\label{3.6}
			\beta \in \left[\frac{pq}{q-2},p+p_0\right).
		\end{equation}
		Let
		$$\rho:=\frac{p_0}{p_0-\beta+p}\quad and \quad \rho_*:=\frac{p_0}{\beta-p}.$$
		In view of \eqref{3.6}, we have
		$
		\rho>1$ and $\rho_*>1.
		$
		Using  $\int_{\Omega}\varphi^{p_0}\leq L$, the H\"{o}lder inequality, the Young inequality and Lemma \ref{le3.1},  we conclude that for any $\eta>0$, there exist
		positive constants  $c_1=c_1(\eta,p,\beta,L)$ and $c_2=c_2(\eta,p,q)$  such that
		\begin{align*}
			\int_{\Omega}\varphi^{\beta}\psi^\alpha&=\int_{\Omega}(\varphi^{\frac{p}{2}}\psi^{\frac{\alpha}{2}})^2\varphi^{\beta-p}\nonumber
			\\
			&\leq\|\varphi^{\frac{p}{2}}\psi^{\frac{\alpha}{2}}\|^2_{L^{2\rho}(\Omega)}\cdot\|\varphi^{\beta-p}\|_{L^{\rho_*}(\Omega)}\nonumber
			\\
			&=\|\varphi^{\frac{p}{2}}\psi^{\frac{\alpha}{2}}\|^2_{L^{2\rho}(\Omega)}
			\cdot\|\varphi\|_{L^{p_0}(\Omega)}^{\beta-p}\nonumber
			\\
			&\leq L^{\frac {\beta-p}{p_0}}\|\varphi^{\frac{p}{2}}\psi^{\frac{\alpha}{2}}\|^2_{L^{2\rho}(\Omega)}\nonumber
			\\
			&\leq \frac{\eta}{2}\int_{\Omega}\varphi^{p-2}\psi^\alpha|\nabla
			\varphi|^2
			+\frac{\eta}{2}\int_{\Omega}\varphi^{p}\psi^{\alpha-2}|\nabla\psi|^2
			+c_1\left\{\int_{\Omega}\varphi\right\}^{p}\cdot
			\left\{\sup\limits_{x\in\Omega}\psi(x)\right\}^\alpha\nonumber
			\\
			&\leq \frac{\eta}{2}\int_{\Omega}\varphi^{p-2}\psi^\alpha|\nabla
			\varphi|^2+\frac{\eta}{2}\int_{\Omega}\varphi\frac{|\nabla
				\psi|^q}{\psi^{q-1}}+\frac {\eta}2\int_{\Omega}\varphi^{\frac{qp-2}{q-2}}\psi^{\frac{\alpha q-2}{q-2}}+c_1\left\{\int_{\Omega}\varphi\right\}^{p}
			\left\{\sup\limits_{x\in\Omega}\psi(x)\right\}^\alpha\nonumber
			\\
			&\leq \frac{\eta}{2}\int_{\Omega}\varphi^{p-2}\psi^\alpha|\nabla
			\varphi|^2+\frac{\eta}{2}\int_{\Omega}\varphi\frac{|\nabla
				\psi|^q}{\psi^{q-1}}+\frac {\eta}2\left\{\sup\limits_{x\in\Omega}\psi(x)\right\}^{\frac{2(\alpha-1)}{q-2}}\int_{\Omega}\varphi^{\frac{qp-2}{q-2}}\psi^\alpha\\
			&\quad+c_1\left\{\int_{\Omega}\varphi\right\}^{p}
			\left\{\sup\limits_{x\in\Omega}\psi(x)\right\}^\alpha\nonumber
			\\
			&\leq \frac{\eta}{2}\int_{\Omega}\varphi^{p-2}\psi^\alpha|\nabla
			\varphi|^2+\frac{\eta}{2}\int_{\Omega}\varphi\frac{|\nabla
				\psi|^q}{\psi^{q-1}}+
			\frac{1}{2}\int_{\Omega}\varphi^{\beta}\psi^\alpha
			+c_2\left\{\sup\limits_{x\in\Omega}\psi(x)\right\}^{\frac{2(\alpha-1) (\beta-1)}{(q-2)(\beta-\frac{qp-2}{q-2})}}\int_{\Omega}\varphi\psi^\alpha\\
			&\quad+c_1\left\{\int_{\Omega}\varphi\right\}^{p}
			\left\{\sup\limits_{x\in\Omega}\psi(x)\right\}^\alpha\\
			&\leq \frac{\eta}{2}\int_{\Omega}\varphi^{p-2}\psi^\alpha|\nabla
			\varphi|^2+\frac{\eta}{2}\int_{\Omega}\varphi\frac{|\nabla
				\psi|^q}{\psi^{q-1}}+
			\frac{1}{2}\int_{\Omega}\varphi^{\beta}\psi^\alpha
			+c_2\left\{\sup\limits_{x\in\Omega}\psi(x)+1\right\}^\frac{\alpha q(\beta-p)-2(\beta-1)}{(q-2)\beta-(qp-2)}\int_{\Omega}\varphi\\
			&\quad+c_1\left\{
\int_{\Omega}\varphi
\right\}^{p}
			\left\{
\sup\limits_{x\in\Omega}\psi(x)+1
\right\}^\alpha
		\end{align*}
		and thus
		\begin{align}\label{3.7}
			\int_{\Omega}\varphi^{\beta}\psi^\alpha&\leq \eta\int_{\Omega}\varphi^{p-2}\psi^\alpha|\nabla
			\varphi|^2+\eta\int_{\Omega}\varphi{\psi}^{-q+1}|\nabla  \psi|^q\nonumber\\
			&\quad+c_3\left\{\sup\limits_{x\in\Omega}\psi(x)+1\right\}^\tau\left\{\int_{\Omega}\varphi
			+\left\{\int_{\Omega}\varphi\right\}^{p}\right\}
		\end{align}
		with
		$$\tau:=\frac{ pq(\alpha-1)}{q-2}+1=\sup_{\beta\in\left[\frac{pq}{q-2},p+p_0\right)}\big\{\frac{\alpha q(\beta-p)-2(\beta-1)}{(q-2)\beta-(qp-2)}\big\}.
		$$
		Furthermore, it is observed that  for $\beta \in [1, \frac{qp}{q-2})$, there exists $c_4=c_4(p,q)>0$ such that $$\int_{\Omega}\varphi^{\beta}\psi^\alpha\leq
		\int_{\Omega}\varphi^{\frac{qp}{q-2}}\psi^\alpha+c_4\int_{\Omega}\varphi\psi^\alpha $$
		by the Young inequality. Together with \eqref{3.7}, this completes the proof.
	\end{proof}
	
	\textbf{Remark.}
	Lemma~\ref{le3.2} is the key tool for handling the cross-diffusive term $\phi'(v)u^{p-1}\nabla v\cdot\nabla u$: it enables us to control the mixed integral $\int_\Omega u^\beta v^\alpha$ by the available $u$-related dissipation and the weighted gradient dissipation of $v$ from Lemma~\ref{lemma3.3},
	with parameters chosen precisely to match the degenerate structure.
	\vskip1mm
	
	With the $L^2(\Omega)$-bound for
	$u$ from Lemma \ref{L2bound} in hand,
we now proceed to bootstrap the integrability of $u$ to arbitrary $L^p(\Omega)$ with $p>2$. This is accomplished by demonstrating that the functional
	\begin{equation}\label{3.24}
		\mathcal F(t):=\int_\Omega u^p+\int_\Omega v^{-q+1}|\nabla v|^q
	\end{equation}
	possesses a quasi-energy structure.
	
	\begin{lemma}\label{le3.6}
		Let $\Omega\subset\mathbb{R}^2$ be a smoothly bounded domain, $\alpha\geq1$.  Then for any $p>2$ and $p+2<q<2p$, one can find a constant $C=C(p,q)$ such that 	for all $t>0$
		\begin{equation}\label{3.17}
			\int_\Omega u^p+\int_\Omega v^{-q+1}|\nabla v|^q\leq C.
		\end{equation}
	\end{lemma}
	\begin{proof}
		From \eqref{2.2} and \eqref{1.6}, it follows that
		there exist positive constants $c_1$ and $c_2$ such that
		\begin{equation}\label{3.17*}
			c_1 v^{\alpha}(x,t)\leq\phi(v(x,t))~~\mbox{and}~|\phi'(v(x,t))|\leq c_2 v^{\alpha-1}(x,t)~~\mbox{for}~(x,t)\in\Omega\times(0,\infty).
		\end{equation}
		Testing the first equation in \eqref{1.1} with $u^{p-1}$ and applying Young's inequality yields
		\begin{align}\label{3.16} \frac{1}{p}\frac{d}{dt}\int_{\Omega}u^p&=\int_{\Omega}\Delta(\phi(v) u)u^{p-1}+a\int_{\Omega}u^p-b\int_{\Omega}u^{p+1}\\
			&=-(p-1)\int_{\Omega}\phi(v)u^{p-2}|\nabla u|^2-(p-1)\int_{\Omega}\phi'(v)u^{p-1}\nabla v\cdot\nabla u+a\int_{\Omega}u^p-b\int_{\Omega}u^{p+1}\nonumber\\
			&\leq-\frac{c_1(p-1)}{2}\int_{\Omega}v^\alpha u^{p-2}|\nabla u|^2+\frac{(p-1)c_2^2}{2c_1}\int_{\Omega}u^p v^{\alpha-2}|\nabla v|^2+a\int_{\Omega}u^p-b\int_{\Omega}u^{p+1}\nonumber
		\end{align}
		for all $t>0$.
		
		Taking an evident linear combination of  \eqref{3.16} and \eqref{3.1**} with $q\in[2,2p)$ entails that
		\begin{align}\label{3.18}
			\frac{d}{dt}&\bigg(\int_{\Omega}u^p+\int_{\Omega}v^{-q+1}|\nabla v|^{q}\bigg)+\int_{\Omega}u^p+\int_{\Omega}v^{-q+1}|\nabla v|^{q}+\frac{q}{8(q+\sqrt{2})^2}\int_{\Omega}v^{-q-1}|\nabla v|^{q+2}\nonumber\\
			&\quad+\frac{c_1p(p-1)}{2}\int_{\Omega}u^{p-2}v^\alpha |\nabla u|^2+(q-1)^2\int_\Omega u v^{-q+1}|\nabla v|^q\nonumber\\
			&\leq
			\frac{p(p-1)c_2^2}{2c_1}\int_{\Omega}u^p v^{\alpha-2}|\nabla v|^2+\int_{\Omega}v^{-q+1}|\nabla v|^{q}+(ap+1)\int_{\Omega}u^p-bp\int_{\Omega}u^{p+1}\\
			&\quad  +c_3(q)\bigg(\int_{\Omega}u^{\frac{q+2}{2}}v+\int_\Omega v\bigg)\nonumber\\
			&\leq
			\frac{p(p-1)c_2^2}{2c_1}\int_{\Omega}u^p v^{\alpha-2}|\nabla v|^2+\int_{\Omega}v^{-q+1}|\nabla v|^{q}+c_4\bigg(\int_{\Omega}u+\int_{\Omega}v\bigg)-\frac{3bp}{4}\int_{\Omega}u^{p+1}+c_5\int_{\Omega}u^{\frac{q+2}{2}}\nonumber\\
			&\leq
			\frac{p(p-1)c_2^2}{2c_1}\int_{\Omega}u^p v^{\alpha-2}|\nabla v|^2+\int_{\Omega}v^{-q+1}|\nabla v|^{q}-\frac{bp}{2}\int_{\Omega}u^{p+1}+c_6\bigg(\int_{\Omega}u+\int_{\Omega}v\bigg),\nonumber
		\end{align}
		for some $c_i>0~(i=3,4,5,6)$.
		We invoke Young's inequality to see that
		\begin{equation}
			\begin{split}
				\int_{\Omega}v^{-q+1}|\nabla v|^{q}&=\int_{\Omega}\bigg\{\frac{q}{16(q+\sqrt{2})^2}v^{-q-1}|\nabla v|^{q+2}\bigg\}^{\frac{q}{q+2}}\cdot\bigg\{\big(\frac{16(q+\sqrt{2})^2}{q}\big)^{\frac{q}{q+2}}v^{\frac{2}{q+2}}\bigg\}\\
				&\leq\frac{q}{16(q+\sqrt{2})^2}\int_{\Omega}v^{-q-1}|\nabla v|^{q+2}+c_7(q)\int_{\Omega} v
			\end{split}
		\end{equation}
		and
		\begin{align}\label{3.19}
			&\quad\frac{p(p-1)c_2^2}{2c_1}\int_{\Omega}u^p v^{\alpha-2}|\nabla v|^2\nonumber\\
			&\leq\frac{q}{16(q+\sqrt{2})^2}\int_{\Omega}v^{-q-1}|\nabla v|^{q+2}+c_8(p,q)\int_\Omega u^{\frac{p(q+2)}{q}}v^{\frac{\alpha(q+2)-2}{q}}\\
			&\leq\frac{q}{16(q+\sqrt{2})^2}\int_{\Omega}v^{-q-1}|\nabla v|^{q+2}+c_8(p,q)\|v_0\|_{L^{\infty}(\Omega)}^{\frac{2(\alpha-1)}{q}}\int_\Omega u^{\frac{p(q+2)}{q}}v^\alpha.\nonumber
		\end{align}
		Note that for any fixed $p>2$, one can find $q>1$ such that
		\begin{equation*}
			q\in(p+2,2p)~~\mbox{and}~~
			\frac{p(q+2)}{q}<p+2.
		\end{equation*}
		Hence, by \eqref{3.8}, the application of Lemma \ref{le3.2} to $\beta:=	\frac{p(q+2)}{q}$ and $p_0:=2$ yields
		\begin{align}\label{3.23}
			c_8(p,q)\|v_0\|_{L^{\infty}(\Omega)}^{\frac{2(\alpha-1)}{q}}\int_\Omega u^{\frac{p(q+2)}{q}}v^\alpha
			&\leq c_9(p,q)\int_\Omega u^{\frac{p(q+2)}{q}}v^\alpha\nonumber\\
			&\leq\frac{c_1p(p-1)}{4} \int_{\Omega}u^{p-2}v^\alpha|\nabla
			u|^2+\frac{(q-1)^2}{2} \int_{\Omega}u{v}^{-q+1}|\nabla v|^q\\
			&\quad+c_{10}\left\{\sup\limits_{x\in\Omega}v(x)+1\right\}^{\frac{ pq(\alpha-1)}{q-2}+1}\left\{\int_{\Omega}u+\left(\int_{\Omega}u\right)^{p}\right\}.\nonumber
		\end{align}
		From \eqref{3.18}-\eqref{3.23}, it follows that there exists $c_{11}>0$ such that
		\begin{align*}
			&~~\frac{d}{dt}\bigg(\int_{\Omega}u^p+\int_{\Omega}v^{-q+1}|\nabla v|^{q}\bigg)+\int_{\Omega}u^p+\int_{\Omega}v^{-q+1}|\nabla v|^{q}\nonumber\\
			&\quad+\frac{c_1p(p-1)}{4}\int_{\Omega}u^{p-2}v^\alpha |\nabla u|^2+\frac{(q-1)^2}{2}\int_\Omega u v^{-q+1}|\nabla v|^q+\frac{bp}{2}\int_{\Omega}u^{p+1}\\
			&\leq c_{10}\left\{\sup\limits_{x\in\Omega}v(x)+1\right\}^{\frac{\alpha pq}{q-2}+1}\left\{\int_{\Omega}u+\left(\int_{\Omega}u\right)^{p}\right\}+c_{11}\bigg(\int_{\Omega} v+\int_{\Omega} u\bigg).
			\nonumber
		\end{align*}
		Therefore, due to \eqref{2.1} and \eqref{2.2}, there exists some constant $c_{12}>0$ satisfying
		\begin{equation}\label{3.26*}
			\frac{d}{dt}\mathcal{F}(t)+\mathcal{F}(t)\leq c_{12},
		\end{equation}
		which, via an ordinary differential inequality argument, implies the uniform boundedness of $\mathcal{F}(t)$ and thus leads to \eqref{3.17}.
	\end{proof}
	
	The pointwise gradient estimate $|\nabla v|\le Cv^{1-r}$ below will allow us
	to carry out the Moser iteration and derive the desired $L^\infty$-bound for $u$.
	
	\begin{lemma}\label{lem3.5}
		Let $\Omega\subset\mathbb{R}^{2}$ be a smoothly bounded domain and $r\in(0,1)$. Then there exists $C>0$ such that for all $x\in\Omega$ and $t>0$,
		$$
		|\nabla v(x,t)|\leq C v^{1-r}(x,t).
		$$
	\end{lemma}
	\begin{proof}
		According to the second equation of \eqref{1.1}, we can see that for all $t>0$
		\begin{equation}\label{4.1*}
			\partial_t v^r = \Delta v^r + h(x,t),
		\end{equation}
		where
		$
		h(x,t):= r(1-r) v^{r-2}|\nabla v|^2 - ruv^r.
		$
		By the variation-of-constants formula, we have
		\begin{equation}\label{4.1a}
			v^r(\cdot,t)=e^{t\Delta}v_0^r+\int_0^t e^{(t-s)\Delta}h(\cdot,s)\,ds\quad\;for\; all\; t>0.
		\end{equation}
		Since $r<1$, for any $\theta\ge2$,
		\begin{equation}\label{4.2*}
			\int_\Omega |h|^{\theta}
			\leq 2^\theta\int_\Omega v^{(r-2)\theta}|\nabla v|^{2\theta}+2^{\theta}\int_\Omega u^\theta v^{r\theta}.
		\end{equation}
		Now fix $k > \max\{4,\frac{2}{r}\}$ and choose $\theta\in (2,\frac{k}{2})$. Since $k>4$, we can choose some $p>2$ such that $p+2<k< 2p$. Therefore, Lemma \ref{le3.6} applies with this choice of $p$ and $q=k$, yielding  $$\int_\Omega \frac{|\nabla v|^k}{v^{k-1}}\leq c_1.$$ By H\"older's inequality and \eqref{2.2}, there exist constants $c_2,c_3>0$ such that for all $t>0$,
		
		\begin{align}\label{4.4*}
			2^{\theta}\int_{\Omega} v^{(r-2)\theta}|\nabla v|^{2\theta}
			&= 2^{\theta}\int_{\Omega}\left\{\frac{|\nabla v|^k}{v^{k-1}}\right\}^{\frac{2\theta}{k}}\cdot v^{\frac{2(k-1)\theta}{k}+(r-2)\theta}\nonumber\\
			&\leq 2^{\theta}\left(\int_{\Omega}\frac{|\nabla v|^k}{v^{k-1}}\right)^{\frac{2\theta}{k}}\cdot\left(\int_{\Omega} v^{\left[\frac{2(k-1)\theta}{k}+(r-2)\theta\right]\cdot\frac{k}{k-2\theta}}\right)^{\frac{k-2\theta}{k}}\\
			&\leq 2^{\theta}|\Omega|^{\frac{k-2\theta}{k}}\|v_0\|_{L^{\infty}(\Omega)}^{\frac{2(k-1)\theta}{k}+(r-2)\theta}\left(\int_{\Omega}\frac{|\nabla v|^k}{v^{k-1}}\right)^{\frac{2\theta}{k}}\nonumber\\
			&\leq c_2,\nonumber
		\end{align}
		and
		\begin{equation}\label{4.5*}
			2^{\theta}\int_\Omega u^{\theta}v^{r\theta}\leq 2^{\theta}\|v_0\|_{L^\infty(\Omega)}^{r\theta}\int_\Omega u^{\theta}\leq c_3,
		\end{equation}
		where $\frac{2(k-1)\theta}{k}+(r-2)\theta= \theta(r-\frac{2}{k})>0$.
		
		Applying the smoothing estimates for the Neumann heat semigroup from \cite{Winkler(2010)} to \eqref{4.1a} and \eqref{1.4}, we obtain constants $c_4,c_5>0$ such that for all $t>0$,
		\begin{align}
			\|\nabla v^r(\cdot,t)\|_{L^{\infty}(\Omega)}
			&= \left\| \nabla e^{t\Delta}  v_0^r + \int_0^t \nabla e^{(t-s)\Delta} h(\cdot,s)\,ds \right\|_{L^{\infty}(\Omega)} \nonumber\\
			&\leq c_4 \|\nabla  v_0^r
			\|_{L^{\infty}(\Omega)} + c_4 \int_0^t (1+(t-s)^{-\frac{1}{2}-\frac{1}{\theta}}) e^{-\lambda(t-s)} \|h(\cdot,s)\|_{L^{\theta}(\Omega)}\,ds \\
			&\leq c_4 \|\nabla  v_0^r
			\|_{L^{\infty}(\Omega)}
			+ c_4 (c_2+c_3)^{\frac 1 {\theta} } \int_0^t (1+(t-s)^{-\frac{1}{2}-\frac{1}{\theta}}) e^{-\lambda(t-s)}\,ds \nonumber\\
			&\leq c_5,\nonumber
		\end{align}
		due to $\int_0^{\infty} \sigma^{-\frac{1}{2}-\frac{1}{\theta}} e^{-\lambda\sigma}d \sigma<\infty$ and $\|\nabla v_0^r\|_{L^\infty(\Omega)}<\infty$.
		Thus $\|\nabla v^r(\cdot,t)\|_{L^\infty(\Omega)}\le c_5$ for all $t>0$, which yields
		$$
		|\nabla v(x,t)| = r^{-1} v^{1-r}(x,t)|\nabla v^r(x,t)|
		\le \frac{c_5}{r} v^{1-r}(x,t).
		$$
		This completes the proof.
	\end{proof}
	
	With the above preparation, we are now in a position to establish the announced $L^\infty(\Omega)$-bound for $u$ by means of the following Moser-type iterative argument. Our procedure will particularly involve a modified recursive inequality from Lemma 6.3 of  \cite{WinklerJDE}.
	\begin{lemma}[{\cite[Lemma 6.3]{WinklerJDE}}]\label{lemma3.9}
		Let $\tau\geq1,\rho\geq1,q\geq0$ and $\{N_k\}_{k\in\{1,2,3,...\}}\subset[1,\infty)$ be such that
		\begin{equation*}
			N_{k}\leq \rho^k N_{k-1}^{2+q\cdot 2^{-k}}+\rho^k \cdot\tau^{2^k}~~~~\mbox{for all}~k\geq1.
		\end{equation*}
		Then
		\begin{equation*}
			\liminf\limits_{k\to\infty}N_k^{\frac{1}{2^k}}\leq(2\sqrt{2}\rho^3\tau^{1+\frac{q}{2}}N_0)^{e^{\frac{q}{2}}}.
		\end{equation*}
	\end{lemma}
	We now apply the above recursive inequality to prove the $L^\infty$-bound for $u$.
	
	\begin{lemma}\label{le3.8Moser}
		Let $\Omega\subset\mathbb{R}^2$ be a smoothly bounded domain.  Then there exists a constant $C>0$ such that for all $t>0$,
		\begin{equation*}
			\|u(\cdot,t)\|_{L^\infty(\Omega)}\leq C.
		\end{equation*}
		
	\end{lemma}
	\begin{proof}
		For arbitrary $T>0$, set $p_k:=2^k$ for $k\geq0$, and define	
		\begin{equation}\label{3.36}
			N_k(T):=\sup\limits_{0<t\leq T}\int_\Omega u^{p_k}(\cdot,t)+1,\qquad k\in\{0,1,2,\cdots\}.
		\end{equation}
		By Proposition \ref{pro1.1}, the solution is classical on $[0,T]$, whence $N_k(T)<\infty$ for each fixed $k$ and $T$. Moreover, according to Lemma \ref{le2.4}, $N_0(T)\le m+1$ with $m$ independent of $T$.
		
		We shall derive a $T$-independent recursive inequality for $N_k(T)$  for $k\geq1$; the desired  uniform bound
		will then follow by letting $T\to\infty$.
		To this end, we recall from \eqref{1.6} and \eqref{2.2} that
		\begin{equation*}
			c_1 v^{\alpha}(x,t)\leq\phi(v(x,t))~~\mbox{and}~~|\phi'(v(x,t))|\leq c_2 v^{\alpha-1}(x,t)~~\mbox{for}~(x,t)\in\Omega\times(0,\infty)
		\end{equation*}
		for some constants $c_1>0$ and $c_2>0$.
		
		Testing the first equation in \eqref{1.1} with $p_k u^{p_k-1}$, we  arrive at
		\begin{align*}
			\frac{d}{dt}\int_{\Omega}u^{p_k}+b\int_{\Omega}u^{p_k}
			&=p_k\int_{\Omega}\Delta(\phi(v) u)u^{p_k-1}+(ap_k+b)\int_{\Omega}u^{p_k}-bp_k\int_{\Omega}u^{p_k+1}\nonumber\\
			&\leq-c_1p_k(p_k-1)\int_{\Omega}u^{p_k-2}v^\alpha|\nabla u|^2
			+c_2 p_k(p_k-1)\int_{\Omega}v^{\alpha-1}u^{p_k-1}|\nabla v||\nabla u|\nonumber\\
			&\quad+p_k(a+b)\int_{\Omega}u^{p_k}-bp_k\int_{\Omega}u^{p_k+1}\\
			&\leq-\frac{c_1p_k(p_k-1)}{2}\int_{\Omega}u^{p_k-2}v^\alpha|\nabla u|^2+\frac{p_k(p_k-1)c_2^2}{2c_1}\int_{\Omega}v^{\alpha-2}u^{p_k}|\nabla v|^2\\
			&\quad+p_k(a+b)\int_{\Omega}u^{p_k}-bp_k\int_{\Omega}u^{p_k+1}.\nonumber
		\end{align*}
		By Lemma \ref{lem3.5}, there exists $c_\tau>0$ such that we have
		\begin{align}			&\frac{d}{dt}\int_{\Omega}u^{p_k}+b\int_{\Omega}u^{p_k}+\frac{c_1}{4}\int_\Omega\big|\nabla(u^{\frac{p_k}{2}}
			v^{\frac{\alpha}{2}})\big|^2\\
			\leq&\big(\frac{p_k(p_k-1)c_2^2}{2c_1}+\frac{c_1\alpha^2}{4}\big)\int_{\Omega}u^{p_k}v^{\alpha-2}|\nabla v|^2
			+p_k(a+b)\int_{\Omega}u^{p_k}-bp_k\int_{\Omega}u^{p_k+1}\nonumber\\
			\leq&c_{\tau}\big(\frac{p_k(p_k-1)c_2^2}{2c_1}+\frac{c_1\alpha^2}{4}\big)\int_{\Omega}u^{p_k}v^{\alpha-2r}
			+p_k(a+b)\int_{\Omega}u^{p_k}-bp_k\int_{\Omega}u^{p_k+1}\nonumber
		\end{align}
		for all $r\in(0,1)$.
		Since $p_k(p_k-1)\leq p_k^2$ and $\frac{p_k(p_k-1)}{2}\geq\frac{p_k^2}{4}$,  there exists a constant $c_3>0$ independent of $k$ and $T$ such that
		\begin{equation}\label{3.37}
			\begin{split}
				&\frac{d}{dt}\int_\Omega u^{p_k}+b\int_\Omega u^{p_k}+\frac{c_1}{4}\int_\Omega\big|\nabla(u^{\frac{p_k}{2}}v^{\frac{\alpha}{2}})\big|^2\\
				\leq &c_3p_k^2\int_{\Omega}u^{p_k}v^{\alpha-2r}+p_k(a+b)\int_\Omega u^{p_k}-bp_k\int_{\Omega}u^{p_k+1}.
			\end{split}
		\end{equation}
		By H\"{o}lder's inequality and the choice $0<r<\min\{1,\frac{\alpha}{4}\}$, we have
		\begin{align}\label{4.10*}
			\int_{\Omega} u^{p_k} v^{\alpha-2r}
			&= \int_{\Omega} \left(u^{\frac{p_k}{2}}v^{\frac{\alpha}{2}}\right)^{\frac{2(\alpha-2r)}{\alpha}} \cdot u^{\frac{2r p_k}{\alpha}}\nonumber \\
			&\leq \left(\int_{\Omega} \left(u^{\frac{p_k}{2}}v^{\frac{\alpha}{2}}\right)^{\frac{2(\alpha-2r)}{\alpha-4r}}\right)^{\frac{\alpha-4r}{\alpha}} \cdot \left(\int_{\Omega} u^{\frac{p_k}{2}}\right)^{\frac{4r}{\alpha}} \\
			&\leq N_{k-1}^{\frac{4r}{\alpha}}(T) \cdot \left\| u^{\frac{p_k}{2}}v^{\frac{\alpha}{2}} \right\|_{L^{\frac{2(\alpha-2r)}{\alpha-4r}}(\Omega)}^{\frac{2(\alpha-2r)}{\alpha}},\nonumber
		\end{align}
		where, by the definition of $N_k(T)$ in \eqref{3.36},
		\begin{equation*}
			\begin{split}
				\int_\Omega u^{\frac{p_k}{2}}=\int_\Omega u^{p_{k-1}}\leq N_{k-1}(T).
			\end{split}
		\end{equation*}
		Here, the condition $\alpha>4r$ allows us to invoke the Gagliardo--Nirenberg inequality, which guarantees the existence of a constant $c_4>0$ independent of $k$ and $T$ such that
		\begin{equation}
			\begin{split}
				\left\| u^{\frac{p_k}{2}}v^{\frac{\alpha}{2}} \right\|_{L^{\frac{2(\alpha-2r)}{\alpha-4r}}(\Omega)}^{\frac{2(\alpha-2r)}{\alpha}}
				\leq c_4 \left\| \nabla\left(u^{\frac{p_k}{2}}v^{\frac{\alpha}{2}}\right) \right\|_{L^2(\Omega)} \cdot \left\| u^{\frac{p_k}{2}}v^{\frac{\alpha}{2}} \right\|_{L^1(\Omega)}^{\frac{\alpha-4r}{\alpha}} + c_4 \left\| u^{\frac{p_k}{2}}v^{\frac{\alpha}{2}} \right\|_{L^1(\Omega)}^{\frac{2(\alpha-2r)}{\alpha}}.
			\end{split}
		\end{equation}
		Moreover, by the H\"{o}lder inequality,
		\begin{equation}\label{4.12*}
			\left\| u^{\frac{p_k}{2}}v^{\frac{\alpha}{2}} \right\|_{L^1(\Omega)} \leq \|v_0\|_{L^{\infty}(\Omega)}^{\frac{\alpha}{2}} \cdot N_{k-1}(T).
		\end{equation}
		Hence it follows from \eqref{4.10*}--\eqref{4.12*} that
		\begin{align}
			c_3 p_k^2 \int_{\Omega} u^{p_k} v^{\alpha-2r}
			&\leq c_3 c_4 p_k^2 N_{k-1}^{\frac{4r}{\alpha}}(T) \cdot \left\| \nabla\left(u^{\frac{p_k}{2}}v^{\frac{\alpha}{2}}\right) \right\|_{L^2(\Omega)} \cdot \left\| u^{\frac{p_k}{2}}v^{\frac{\alpha}{2}} \right\|_{L^1(\Omega)}^{\frac{\alpha-4r}{\alpha}} \nonumber\\
			&\quad + c_3 c_4 p_k^2 N_{k-1}^{\frac{4r}{\alpha}}(T) \cdot \left\| u^{\frac{p_k}{2}}v^{\frac{\alpha}{2}} \right\|_{L^1(\Omega)}^{\frac{2(\alpha-2r)}{\alpha}}\nonumber\\
			&= \left\{ \frac{c_1}{4}\int_{\Omega}\left|\nabla\left(u^{\frac{p_k}{2}}v^{\frac{\alpha}{2}}\right)\right|^2 \right\}^{\frac{1}{2}} \cdot \left(\frac{4}{c_1}\right)^{\frac{1}{2}} c_3 c_4 p_k^2 N_{k-1}^{\frac{4r}{\alpha}}(T) \cdot \left\| u^{\frac{p_k}{2}}v^{\frac{\alpha}{2}} \right\|_{L^1(\Omega)}^{\frac{\alpha-4r}{\alpha}}\nonumber \\
			&\quad + c_3 c_4 p_k^2 N_{k-1}^{\frac{4r}{\alpha}}(T)  \cdot \left\| u^{\frac{p_k}{2}}v^{\frac{\alpha}{2}} \right\|_{L^1(\Omega)}^{\frac{2(\alpha-2r)}{\alpha}}\\
			&\leq \frac{c_1}{4}\int_{\Omega}\left|\nabla\left(u^{\frac{p_k}{2}}v^{\frac{\alpha}{2}}\right)\right|^2 + \left(\frac{4}{c_1}(c_3 c_4)^{2}\right) p_k^{4} N_{k-1}^{\frac{8r}{\alpha}}(T)  \cdot \left\| u^{\frac{p_k}{2}}v^{\frac{\alpha}{2}} \right\|_{L^1(\Omega)}^{\frac{2\alpha-8r}{\alpha}}\nonumber\\
			&\quad + c_3 c_4 p_k^2 N_{k-1}^{\frac{4r}{\alpha}}(T)  \cdot \left\| u^{\frac{p_k}{2}}v^{\frac{\alpha}{2}} \right\|_{L^1(\Omega)}^{\frac{2(\alpha-2r)}{\alpha}}\nonumber\\
			&\leq  \frac{c_1}{4}\int_{\Omega}\left|\nabla\left(u^{\frac{p_k}{2}}v^{\frac{\alpha}{2}}\right)\right|^2 + c_5 p_k^{4} N_{k-1}^{2}(T)  \cdot \Big\{ \|v_0\|_{L^{\infty}(\Omega)}^{\alpha-4r} + \|v_0\|_{L^{\infty}(\Omega)}^{\alpha-2r} \Big\}\nonumber \\
			&\leq \frac{c_1}{4}\int_{\Omega}\left|\nabla\left(u^{\frac{p_k}{2}}v^{\frac{\alpha}{2}}\right)\right|^2 + c_6p_k^{4} N_{k-1}^{2}(T),\nonumber
		\end{align}
		where $c_5,c_6>0$ are independent of $k$ and $T$.
		It follows from \eqref{3.37} that
		\begin{align*}
			\frac{d}{dt}\int_\Omega u^{p_k}+b\int_\Omega u^{p_k}&\leq c_{6}p_k^4N_{k-1}^{2}+p_k(a+b)\int_\Omega u^{p_k}-bp_k\int_{\Omega}u^{p_k+1}\\
			&\leq c_{7}p_k^{4} N_{k-1}^{2}+p_k\int_{\{u\geq\frac{a+b}{b}\}} u^{p_k}(a+b-bu)+p_k\int_{\{u\leq\frac{a+b}{b}\}}u^{p_k}(a+b-bu)\\
			&\leq c_{7}p_k^{4} N_{k-1}^{2}+p_k(a+b)\int_{\{u\leq\frac{a+b}{b}\}}u^{p_k}\\
			&\leq c_{7}p_k^{4} N_{k-1}^{2}(T)+p_k|\Omega|(a+b)\bigg(\frac{a+b}{b}\bigg)^{p_k}.
		\end{align*}
		Hence, integrating this differential inequality over $(0,t)$ yields
		\begin{align}
			\int_\Omega u^{p_k}(\cdot,t)\leq \int_\Omega u_0^{p_k}+ \frac{c_{7}}{b}p_k^{4} N_{k-1}^{2}(T)+p_k|\Omega|\bigg(\frac{a+b}{b}\bigg)^{p_k+1}.
		\end{align}
		Utilizing the definition of $N_k(T)$ provided in  \eqref{3.36}, we obtain that
		\begin{align}
			N_k(T)&\leq1+\int_\Omega u_0^{p_k}+p_k|\Omega|\bigg(\frac{a+b}{b}\bigg)^{p_k+1}+ c_{8}p_k^{4} N_{k-1}^{2}(T)\nonumber\\
			&\leq p_k+p_k|\Omega|\|u_0\|_{L^\infty(\Omega)}^{p_k}+p_k|\Omega|\bigg(\frac{a+b}{b}\bigg)^{p_k+1}+ c_{8}p_k^{4} N_{k-1}^{2}(T)
		\end{align}
		with $c_8:=(\frac{c_7}{b}+1)$ independent of $k,T$. It is observed that
		$$1+|\Omega|\|u_0\|_{L^\infty(\Omega)}^{p_k}+|\Omega|\bigg(\frac{a+b}{b}\bigg)^{p_k+1} \leq \tau^{p_k}$$
		with
		$$\tau:=1+(|\Omega|+1)\|u_0\|_{L^\infty(\Omega)}+(|\Omega|+1)\bigg(\frac{a+b}{b}\bigg)^2,$$
		and
		\begin{align*}
			c_{8}p_k^{4}&=c_{8}2^{4k}
			=e^{\ln c_{8}+4k\ln 2}
			\leq e^{(\ln c_{8}+4\ln2)k}
			=\big\{e^{(\ln c_{8}+4\ln2)}\big\}^k
		\end{align*}
		with $\rho:=\max\{2,e^{(\ln c_{8}+4\ln2)}\}$,
		this entails that
		\begin{equation*}
			N_k(T)\leq\rho^k\tau^{p_{k}}+\rho^k N_{k-1}^{2}(T).
		\end{equation*}
		
		Applying Lemma \ref{lemma3.9} with $q=0$ yields
		$$
		\|u(\cdot,t)\|_{L^\infty(\Omega)}
		\le \liminf_{k\to\infty} N_k^{1/2^k}(T)
		\le 2\sqrt{2}\rho^3\tau N_0(T)
		\quad\text{for all }0<t\le T.
		$$
		Since $T>0$ is arbitrary, letting $T\to\infty$ yields the desired uniform bound of $u$ for all $t>0$ and hence the proof is  complete.
	\end{proof}
	
	At this point, we combine the preceding results to prove Theorem \ref{theorem1.1}.
	\begin{proof}[Proof of Theorem \ref{theorem1.1}]
		According to Lemma \ref{le3.8Moser}, there exists constant $c_1>0$ such that
		\begin{equation*}
			\| u(\cdot,t)\|_{L^{\infty}(\Omega)}\leq c_1~~\mbox{for all}~~t>0.
		\end{equation*}
		Subsequently, by smoothing estimates for the Neumann heat semigroup (see Lemma \ref{le2.5}), we obtain
		\begin{equation}\label{3.31}\|\nabla v(\cdot,t)\|_{L^\infty(\Omega)}\leq c_2~~\mbox{for all}~~t>0.\end{equation} This completes the proof of Theorem \ref{theorem1.1}.
	\end{proof}
	
\section{Large time behavior of solutions}
This section is devoted to the proof of Theorem~\ref{theorem1.2}.
By combining heat semigroup estimates with a reciprocal-power functional, we first establish the qualitative \(L^2\)-convergence of \(u\) to \(\frac{a}{b}\), which in turn induces the exponential decay of \(v\).
Utilizing this decay in the functional estimate yields the exponential \(L^2\)-convergence of \(u\), which is finally upgraded to \(L^\infty\)-bound of \(u\) by a Moser-type iteration.
\begin{lemma}\label{lem4.1}
Under the assumptions of Theorem~\ref{theorem1.2}, we have
\begin{equation}\label{4.1}
\|v(\cdot,t)\|_{L^\infty(\Omega)}\to0\quad\mbox{as }t\to\infty.
\end{equation}
Moreover, there exists a constant $C>0$ such that
\begin{equation}\label{4.5}
	\|\nabla v(\cdot,t)\|_{L^\infty(\Omega)}
	\le C\inf_{x\in\Omega}v(x,t)
	\qquad\mbox{for all }t\ge1.
\end{equation}
\end{lemma}
\begin{proof}
By Theorem~\ref{theorem1.1}, there exists \(M>0\)
	such that \(0<u(x,t)\le M\) for \(x\in \Omega,t>0\).
	Since \(v>0\), the second equation in \eqref{1.1} implies
	\[
	-M v\le v_t-\Delta v=-uv\le0.
	\]
	For each \(t\ge1\), set \(f:=v(\cdot,t-1)\).
	The parabolic comparison principle gives
	\begin{equation}\label{eq:4.2}
		e^{-M}(e^\Delta f)(x)
		\le v(x,t)\le (e^\Delta f)(x) \quad\mbox{for } x \in \Omega.
	\end{equation}
	Applying Lemma~\ref{lem:heat_kernel_bound} with \(\tau=1\), we infer from \eqref{eq:4.2} that
	\begin{equation}\label{eq:4.3}
		e^{-M}C_1(1)\|f\|_{L^1(\Omega)}
		\le v(x,t)
		\le C_2(1)\|f\|_{L^1(\Omega)}
		\quad\mbox{for all }x\in\overline\Omega.
	\end{equation}
It follows that, with \(c_1:=e^{M}C_2(1)/C_1(1)\),
	\begin{equation}\label{4.2}
		\|v(\cdot,t)\|_{L^\infty(\Omega)}
		\le c_1\inf_\Omega v(\cdot,t)
		\quad\mbox{for all }t\ge1.
	\end{equation}

To establish \eqref{4.1}, we argue by contradiction.
	If \eqref{4.1} is false, the nonnegativity and monotonicity of
	$\|v(\cdot,t)\|_{L^\infty(\Omega)}$ imply that
	\(V_\infty := \lim\limits_{t\to\infty}\|v(\cdot,t)\|_{L^\infty(\Omega)}>0\).
	By the definition of the limit, there exists $t_0\ge1$ such that
	\[
	\|v(\cdot,t)\|_{L^\infty(\Omega)}
	\ge\frac{1}{2}V_\infty
	\qquad\text{for all }t\ge t_0.
	\]
	Together with \eqref{4.2}, this yields
	\[
	v(x,t)
	\ge\frac1{c_1}\|v(\cdot,t)\|_{L^\infty(\Omega)}
	\ge\frac{V_\infty}{2c_1}
	=:c_2>0
	\]
	for all $x\in\overline\Omega$ and $t\ge t_0$.
	Integrating the second equation in \eqref{1.1}
	over $\Omega\times(t_0,T)$, we obtain
\begin{align}\label{4.6}
		c_2\int_{t_0}^T\int_\Omega u
		\le\int_{t_0}^T\int_\Omega uv
		=\int_\Omega v(\cdot,t_0)-\int_\Omega v(\cdot,T)
		\le\int_\Omega v(\cdot,t_0)
		\qquad\text{for all }T>t_0.
	\end{align}
Testing the same equation by $v$ and integrating in time
	also gives
	\begin{equation}\label{4.7a}
		\int_{t_0}^\infty\int_\Omega|\nabla v|^2
		\le\frac12\int_\Omega v^2(\cdot,t_0)<\infty.
	\end{equation}
On the other hand, testing the first equation in \eqref{1.1} by \(1/u\) and applying Young's inequality yields
	\begin{align*}
		\frac{d}{dt}\int_\Omega\ln u
		&=a|\Omega|-b\int_\Omega u
		+\int_\Omega\phi(v)|\nabla\ln u|^2
		+\int_\Omega\phi'(v)\nabla v\cdot\nabla\ln u\\
		&\ge a|\Omega|-b\int_\Omega u
		-\frac14\int_\Omega
		\frac{|\phi'(v)|^2}{\phi(v)}|\nabla v|^2.
	\end{align*}
Integrating the above inequality over \((t_0, T)\) and invoking \eqref{4.6}, \eqref{4.7a}, along with the uniform boundedness of \(\frac{|\phi'(v)|^2}{\phi(v)}\) for \(c_2 \le v \le \|v_0\|_{L^\infty(\Omega)}\), we find \(c_3 > 0\) such that
	\[
	\int_\Omega\ln u(\cdot,T)
	\ge a|\Omega|(T-t_0)-c_3
	\qquad\text{for all }T>t_0.
	\]
	Since the right-hand side tends to infinity linearly as
	$T\to\infty$, this contradicts the uniform upper bound
	$
	\int_\Omega\ln u(\cdot,T)\le|\Omega|\ln M.
	$
	Hence \eqref{4.1} follows.
	
	Finally, we prove inequality \eqref{4.5}.
	For \(t\ge1\), since \(v_t\le\Delta v\), comparison on
	\([t-1,t-\frac12]\) gives
	\[
	v(\cdot,t-\tfrac12)\le e^{\frac12\Delta}v(\cdot,t-1).
	\]

Combining this with Lemma~\ref{lem:heat_kernel_bound} (with \(\tau=\tfrac12\)) and the lower bound in \eqref{eq:4.3},
	we obtain
	\begin{equation}\label{4.8}
		\|v(\cdot,t-\tfrac12)\|_{L^\infty(\Omega)}
		\le\|e^{\frac12\Delta}v(\cdot,t-1)\|_{L^\infty(\Omega)}
		\le C_2(\tfrac12)\|v(\cdot,t-1)\|_{L^1(\Omega)}\le c_4\inf_{\Omega}v(\cdot,t),
	\end{equation}
	where \(c_4 := \frac{C_2(1/2) e^{M}}{C_1(1)}\).
	We now apply the variation-of-constants formula
	on $[t-\frac12,t]$ and the Neumann heat semigroup
	gradient estimates in \cite[Lemma~1.3]{Winkler(2010)}.
	Using $u\le M$, \eqref{4.8} and the monotonicity of
	$\|v(\cdot,t)\|_{L^\infty(\Omega)}$, we obtain
	\begin{align*}
		\|\nabla v(\cdot,t)\|_{L^\infty(\Omega)}
		&\le
		c_5\|v(\cdot,t-\tfrac12)\|_{L^\infty(\Omega)}+
		c_5\int_{t-1/2}^t(t-\xi)^{-1/2}
		\|u(\cdot,\xi)v(\cdot,\xi)\|_{L^\infty(\Omega)}
		\,d\xi\\
		&\le
		c_{6}\|v(\cdot,t-\tfrac12)\|_{L^\infty(\Omega)}\\
		&\le c_{7}\inf_\Omega v(\cdot,t)
	\end{align*}
for all $t\ge1$. This establishes \eqref{4.5} and completes the proof.
\end{proof}

To derive convergence estimates from the reciprocal-power functional,
we first establish the following inequalities.
\begin{lemma}\label{lem:Psi}
	Let
	\[
	\Psi(s):=s^{-4}+4s-5,\qquad s>0.
	\]
	Then the following estimates hold:
	\begin{enumerate}
		\item[(i)] For every $s_0\geq1$, there exists \(C(s_0)>0\) such that
		\begin{equation}\label{eq:Psi-coercivity}
			(s-1)^2\leq C(s_0)\Psi(s)
			\qquad\text{for all }0<s\leq s_0.
		\end{equation}
		\item[(ii)]There exists \(C>0\) such that
		\begin{equation}\label{eq:Psi-reciprocal}
			(s^{-1}-1)^2\leq C\Psi(s)
			\qquad\text{for all }s>0.
		\end{equation}
	\end{enumerate}
\end{lemma}

\begin{proof}
	Since $\Psi(1)=\Psi'(1)=0$ and, for each $s_0\ge1$,
	\[
	\Psi''(s)=20s^{-6}\ge c_1
	\qquad\text{for all }0<s\le s_0
	\]
	with some $c_1:=c_1(s_0)>0$, integrating twice   yields
	\[
	\Psi(s)\ge \frac{c_1}{2}(s-1)^2
	\qquad\text{for all }0<s\le s_0.
	\]
	This proves \eqref{eq:Psi-coercivity}. To establish \eqref{eq:Psi-reciprocal}, we set \(z=s^{-1}\) and
	\[
	H(z):=\Psi(z^{-1})=z^4+4z^{-1}-5.
	\]
	Then \(H(1)=H'(1)=0\), and
	\[
	H''(z)=12z^2+8z^{-3}\ge c_2
	\qquad\text{for all }z>0
	\]
	with some \(c_2>0\). Integrating twice again, we obtain
	\[
	H(z)\ge \frac{c_2}{2}(z-1)^2
	\qquad\text{for all }z>0,
	\]
which completes the proof.
	\end{proof}

We next apply the reciprocal-power functional to obtain qualitative convergence
of $u$. The relative gradient estimate in \eqref{4.5} controls its
cross-diffusion term throughout the range $\alpha\ge1$.

\begin{lemma}\label{lemma4.3}
Under the assumptions of Theorem~\ref{theorem1.2}, we have
\begin{equation}\label{4.9}
u(\cdot,t)\to \frac{a}{b}
\quad\mbox{in }L^2(\Omega)\quad\mbox{as }t\to\infty.
\end{equation}
\end{lemma}
\begin{proof}
	Set $u_*:=a/b$, $s:=u/u_*$, and
	\[
	E(t):=\int_\Omega \Psi(s(\cdot,t)),
	\]
	where $\Psi(\cdot)$ is defined in Lemma~\ref{lem:Psi}. By the positivity of $u(\cdot,t)$ for all $t>0$, the functional $E(t)$ is well defined on $(0,\infty)$.   A direct calculation gives
		\begin{equation}\label{4.10}
			4(s^{-4}-s)(1-s)-2\Psi(s)
			=\frac{2(s-1)^4(2s^2+2s+1)}{s^4}\geq0.
		\end{equation}
		Since $s_t=u_t/u_*$, we get
		\begin{align}\label{4.11}
			E'(t)
			&=\int_\Omega\frac{\Psi'(s)}{u_*}\Delta(\phi(v)u)
			+a\int_\Omega\Psi'(s)s(1-s)\nonumber\\
			&=\int_\Omega\frac{\Psi'(s)}{u_*}\Delta(\phi(v)u)-4a\int_\Omega(s^{-4}-s)(1-s)\\
			&\leq\int_\Omega\frac{\Psi'(s)}{u_*}\Delta(\phi(v)u)-2aE(t)\nonumber
		\end{align}
		by \eqref{4.10}.
		On the other hand, by integration by parts, we have
		\begin{align}\label{4.12}
			\int_\Omega\frac{\Psi'(s)}{u_*}\Delta(\phi(v)u)=&-\int_\Omega\frac{\Psi''(s)}{u_*^2}\phi(v)|\nabla u|^2
			-\int_\Omega\frac{\Psi''(s)s}{u_*}\phi'(v)\nabla u\cdot\nabla v\nonumber\\
			=&-\int_\Omega\frac{\Psi''(s)\phi(v)}{u_*^2}
			\left|\nabla u+\frac{u_*s\phi'(v)}{2\phi(v)}\nabla v\right|^2+\frac14\int_\Omega\Psi''(s)s^2
			\frac{|\phi'(v)|^2}{\phi(v)}|\nabla v|^2\nonumber\\
			\leq& 5\int_\Omega s^{-4}
			\frac{|\phi'(v)|^2}{\phi(v)}|\nabla v|^2.
		\end{align}
To control the last term, we set
\begin{equation*}
G(t):=\left\|
\frac{|\phi'(v(\cdot,t))|^2}{\phi(v(\cdot,t))}
|\nabla v(\cdot,t)|^2
\right\|_{L^\infty(\Omega)}.
\end{equation*}
Since $\int_\Omega s^{-4}\le E(t)+5|\Omega|$,
the combination of \eqref{4.11}  and \eqref{4.12} leads to
\begin{equation}\label{4.20}
E'(t)+(2a-5G(t))E(t)\le25|\Omega|G(t)
\quad\mbox{for all }t>0.
\end{equation}
By \eqref{1.5}, \eqref{1.6} and \eqref{2.2}, there exist $c_{1},c_{2}>0$ such that
\begin{equation}\label{4.4}
c_{1}v^\alpha\le\phi(v)\le c_{2}v^\alpha,
\qquad |\phi'(v)|\le c_{2}v^{\alpha-1}.
\end{equation}
Combining \eqref{4.4} with \eqref{4.5} and \eqref{4.1}, we obtain
\begin{equation}\label{4.13a}
G(t)\le c_{3}\|v^{\alpha-2}|\nabla v|^2\|_{L^\infty(\Omega)}
\le c_{4}\|v^\alpha\|_{L^\infty(\Omega)}\le c_{4}\|v(\cdot,t)\|_{L^\infty(\Omega)}^\alpha\to0
\quad\mbox{as }t\to\infty.
\end{equation}
By \eqref{4.13a}, there exists $T_0\ge1$ such that $2a-5G(t)\ge a$
for $t\ge T_0$. Since $E(t)\ge0$, \eqref{4.20} implies
$E'(t)+aE(t)\le25|\Omega|G(t)$ for $t\ge T_0$.
Integration against the factor $e^{at}$ therefore yields
\[
E(t)\le E(T)e^{-a(t-T)}
+\frac{25|\Omega|}{a}\sup_{\xi\ge T}G(\xi)\quad\mbox{for }
T_0\le T\le t.
\]
Letting first $t\to\infty$ and then $T\to\infty$, we obtain $E(t)\to0$. Furthermore, by Theorem~\ref{theorem1.1},
		$
		\sup\limits_{t>0}\|u(\cdot,t)\|_{L^\infty(\Omega)}<\infty$.
		Therefore, we set
		\[
		c_{5}:=\max\left\{1,\frac{1}{u_*}
		\sup_{t>0}\|u(\cdot,t)\|_{L^\infty(\Omega)}\right\}.
		\]
		Then \(0<s(x,t)\leq c_{5}\) for all
		\(x\in\overline{\Omega}\) and \(t>0\). On the other hand, applying Lemma~\ref{lem:Psi}(i) and the definition of $E(t)$, there exists a constant $c_6>0$ such that
		\[
		\|u(\cdot,t)-u_*\|_{L^2(\Omega)}^2
		=u_*^2\int_\Omega (s(\cdot,t)-1)^2
		\leq c_6E(t)\to0
		\qquad\text{as }t\to\infty.
		\]
		This proves the asserted convergence of \eqref{4.9}.
\end{proof}

The \(L^2(\Omega)\)-convergence of \(u\) established in
Lemma~\ref{lemma4.3} yields exponential decay of \(v\) in \(L^2(\Omega)\).
The estimates \eqref{4.2} and \eqref{4.5} then give exponential decay
in \(W^{1,\infty}(\Omega)\) with the same rate.

\begin{lemma}\label{lemma4.4}
Under the assumptions of Theorem~\ref{theorem1.2}, there exists $C>0$ such that
\begin{equation}\label{4.18}
\|v(\cdot,t)\|_{W^{1,\infty}(\Omega)}
\le Ce^{-\frac{a}{2b} t}
\quad\mbox{for all }t>0.
\end{equation}
\end{lemma}
\begin{proof}
Testing the second equation in \eqref{1.1} by $v$ yields
\begin{equation}\label{4.13}
\frac12\frac{d}{dt}\int_\Omega v^2
+\int_\Omega|\nabla v|^2+\frac ab\int_\Omega v^2
=-\int_\Omega(u-\frac ab)v^2.
\end{equation}
Set $\varepsilon(t):=\|u(\cdot,t)-\frac{a}{b}\|_{L^2(\Omega)}$.
By H\"older's inequality and the two-dimensional Gagliardo--Nirenberg
inequality, we have
\begin{align}\label{4.14}
\left|\int_\Omega(u-\frac ab)v^2\right|
&\le c_{1}\varepsilon(t)
\big(\|\nabla v\|_{L^2(\Omega)}\|v\|_{L^2(\Omega)}
+\|v\|_{L^2(\Omega)}^2\big)\nonumber\\
&\le\frac12\|\nabla v\|_{L^2(\Omega)}^2
+\left(\frac{c_{1}^2\varepsilon(t)^2}{2}
+c_{1}\varepsilon(t)\right)\|v\|_{L^2(\Omega)}^2.
\end{align}
Since $\varepsilon(t)\to0$ by Lemma~\ref{lemma4.3}, the coefficient in
parentheses is at most $\frac{a}{2b}$ for all sufficiently large $t$.
Combining \eqref{4.13} and \eqref{4.14}, we obtain
\[
\frac{d}{dt}\|v(\cdot,t)\|_{L^2(\Omega)}^2
\le-\frac ab\|v(\cdot,t)\|_{L^2(\Omega)}^2.
\]
Integrating this differential inequality and using the boundedness of
\(v\) on finite time intervals, we find \(c_2>0\) such that
$\|v(\cdot,t)\|_{L^2(\Omega)}\le c_2e^{-\frac{a}{2b} t}$ for all $t>0$.
By \eqref{4.2} and \eqref{4.5}, we then obtain, for $t\ge1$,
\begin{equation*}
	\|v(\cdot,t)\|_{L^\infty(\Omega)}+\|\nabla v(\cdot,t)\|_{L^\infty(\Omega)}
	\le c_3\inf_\Omega v(\cdot,t)
    \le c_3|\Omega|^{-1/2}\|v(\cdot,t)\|_{L^2(\Omega)}
    \le c_4e^{-\frac{a}{2b} t}.
\end{equation*}
The estimate on $0<t<1$ follows from
Theorem~\ref{theorem1.1} after enlarging the constant.
\end{proof}

Having established the exponential decay of the signal in Lemma \ref{lemma4.4}, we next apply the reciprocal-power functional to improve the vanishing remainder in \eqref{4.20} to an exponentially decaying one.

\begin{lemma}\label{lem4.5}
Under the assumptions of Theorem~\ref{theorem1.2}, let
$
u_*:=\frac{a}{b}$, $s:=\frac{u}{u_*}$, and
$E(t):=\int_\Omega\Psi(s(\cdot,t)),$
where $\Psi(\cdot)$ is defined in Lemma~\ref{lem:Psi}.
For every $0<\rho<\min\{a,\frac{a\alpha}{4b}\}$, there exists $C_1(\rho)>0$, depending on $\rho$, such that
\begin{equation}\label{4.21*}
	E(t)\le C_1(\rho) e^{-2\rho t}
	\qquad\text{for all }t\ge1.
\end{equation}
In particular, there exists $C_2(\rho)>0$ such that
\begin{equation}\label{4.15}
	\|u(\cdot,t)-u_*\|_{L^2(\Omega)}
	\le C_2(\rho) e^{-\rho t}
	\qquad\text{for all }t>0.
\end{equation}
\end{lemma}
\begin{proof}
Let $G(t)$ be as defined in the proof of Lemma~\ref{lemma4.3}.
By \eqref{4.13a} and Lemma~\ref{lemma4.4},
$G(t)\le c_{1}e^{-\frac{a\alpha}{2b} t}$ for $t\ge1$.
Fix $0<\rho<\min\{a,\frac{a\alpha}{4b}\}$ and choose $T\ge1$ such that
$2a-5G(t)\ge2\rho$ for $t\ge T$. Consequently,
\[
E'(t)+2\rho E(t)
\le25|\Omega|c_1e^{-\frac{a\alpha}{2b} t}
\qquad\text{for all }t\ge T.
\]
Since $\rho<\frac{a\alpha}{4b}$, multiplying \eqref{4.20} by $e^{2\rho t}$ and integrating over $[T,t]$ yields
\[
e^{2\rho t}E(t)
\le e^{2\rho T}E(T)
+25|\Omega|c_1\int_T^t e^{-(\frac{a\alpha}{2b}-2\rho)\xi}\,d\xi
\le c_2
\qquad\text{for all }t\ge T
\]
with some $c_2:=c_2(\rho)>0$. Hence, we have $E(t)\le c_2e^{-2\rho t}$ for $t\ge T$.
Combining this estimate with the boundedness of \(E(t)\) on \([1,T]\), we find a constant $c_3>0$ such that
\begin{equation}\label{4.21}
	E(t)\le c_3e^{-2\rho t}
	\quad\mbox{for all }t\ge1,
\end{equation}
which yields \eqref{4.21*}. Moreover, by Theorem~\ref{theorem1.1}, $s(\cdot,t)$ is uniformly bounded.
Hence, Lemma~\ref{lem:Psi}(i), together with \eqref{4.21},
yields
\[
\left\|u(\cdot,t)-\frac{a}{b}\right\|_{L^2(\Omega)}^2=\frac{a^2}{b^2}\int_\Omega(s(\cdot,t)-1)^2
\leq c_4E(t)
\leq c_5e^{-2\rho t}
\qquad\text{for all }t\geq1,
\]
where $c_4,c_5>0$ are independent of $t$. Together with Theorem~\ref{theorem1.1} for $0<t<1$,
this yields \eqref{4.15} for all $t>0$
after enlarging the constant if necessary.
\end{proof}

To obtain convergence in $L^\infty(\Omega)$, we estimate the upper deviation
of the normalized density and that of its reciprocal separately.
This avoids assuming the uniform positive lower bound for $u$.

\begin{lemma}\label{lem4.6}
Under the assumptions of Theorem~\ref{theorem1.2}, for every
$0<\sigma<\min\{a,\frac{a\alpha}{8b}\}$, there exists $C(\sigma)>0$, depending on $\sigma$, such that
\begin{equation}\label{4.22a}
\bigg\|u(\cdot,t)-\frac{a}{b}\bigg\|_{L^\infty(\Omega)}
\le C(\sigma) e^{-\sigma t}
\quad\mbox{for all }t>0.
\end{equation}
\end{lemma}
\begin{proof}
We use $s:=\frac{u}{u_*}$ as above and write $z:=\frac{1}{s}$.
A direct calculation together with Young's inequality gives
\begin{equation*}
\begin{aligned}
z_t&=\nabla\cdot(\phi(v)\nabla z-\phi'(v)\nabla v\,z)+2\phi'(v)\nabla v\cdot\nabla z
-\frac{2\phi(v)}{z}|\nabla z|^2+a(1-z)\\
&\le\nabla\cdot(\phi(v)\nabla z-\phi'(v)\nabla v\,z)
	+\frac{z|\phi'(v)|^2}{2\phi(v)}|\nabla v|^2+a(1-z).
\end{aligned}
\end{equation*}
For any given $\sigma\in\left(0,\min\left\{a,\frac{a\alpha}{8b}\right\}\right)$,
set $X:=e^{\sigma t}(s-1)$ and $Y:=e^{\sigma t}(z-1)$.
Then we obtain
	\begin{align}
		X_t
		&=\nabla\cdot\left(
		\phi(v)\nabla X+\phi'(v)\nabla v(X+e^{\sigma t})
		\right)
		+(\sigma-as)X,
		\label{4.22}
\end{align}
as well as
\begin{align}
		Y_t
		&\le\nabla\cdot\left(
		\phi(v)\nabla Y-\phi'(v)\nabla v(Y+e^{\sigma t})
		\right)\nonumber\\
		&\quad+
		\left(\sigma-a+
		\frac{|\phi'(v)|^2}{2\phi(v)}|\nabla v|^2\right)Y
		+\frac{e^{\sigma t}|\phi'(v)|^2}{2\phi(v)}
		|\nabla v|^2.
		\label{4.23}
	\end{align}
Moreover, by \eqref{4.4}, \eqref{4.5}, \eqref{4.2} and
Lemma~\ref{lemma4.4}, there exist
$c_1,c_2>0$ such that
\begin{equation}\label{4.26}
	\frac{|\phi'(v)|^2}{\phi(v)}|\nabla v|^2
	\le c_1\|v(\cdot,t)\|_{L^\infty(\Omega)}^\alpha
	\le c_2e^{-\frac{a\alpha}{2b} t}
	\quad\mbox{for all }t\ge1.
\end{equation}
Since $\sigma<a$, we may therefore choose $t_0>1$ sufficiently large to ensure that, for all $t\ge t_0$,
\[
\sigma-a+
\frac{|\phi'(v)|^2}{2\phi(v)}|\nabla v|^2
\le-c_3,
\qquad
\frac{e^{\sigma t}|\phi'(v)|^2}{2\phi(v)}
|\nabla v|^2\le c_4,
\]
where $c_3:=(a-\sigma)/2>0, c_4>0$.
Consequently, \eqref{4.23} implies
	\begin{equation}\label{4.27}
Y_t\le\nabla\cdot\left(
		\phi(v)\nabla Y
		-\phi'(v)\nabla v(Y+e^{\sigma t})
		\right)
		-c_3Y+c_4
	\end{equation}
on $\{Y>0\}$ for $t\ge t_0$.
Similarly, since $s>1$ on $\{X>0\}$, the
zeroth-order term in \eqref{4.22} satisfies
$(\sigma-as)X\le-c_3X$ for $t\ge t_0$.

Fix $t\ge t_0$. By Theorem~\ref{theorem1.1} and Lemma~\ref{lem:Psi}, there exists a constant $c_5>0$ such that
\[
(s-1)_+^2+(z-1)_+^2\le c_5\Psi(s).
\]
Applying Lemma~\ref{lem4.5} with $\rho=\sigma$, there exists a constant $c_6:=c_6(\sigma)>0$ such that
\begin{equation}\label{4.25}
	\begin{aligned}
		\|X_+(\cdot,t)\|_{L^2(\Omega)}^2
		+\|Y_+(\cdot,t)\|_{L^2(\Omega)}^2\le c_5e^{2\sigma t}E(t)\le c_6
		\qquad\text{for all }t\ge t_0.
	\end{aligned}
\end{equation}

	Let $Z$ denote either $X_+$ or $Y_+$.
	For $p\ge2$, we test \eqref{4.22} or \eqref{4.23}
	by $pZ^{p-1}$.
	The tests are justified by smooth approximation of the
	positive part, and the Neumann conditions eliminate
	the boundary terms.
Using \eqref{4.26}, the lower bound
$
\phi(v(x,t))
\ge c_7\|v(\cdot,t)\|_{L^\infty(\Omega)}^\alpha
$ for all $t\ge1$, which follows from \eqref{4.4} and \eqref{4.2},
together with Young's inequality and
	$Z^{p-2}\le1+Z^p$ on $\{Z>0\}$, we obtain
	\begin{align*}
		&\frac{d}{dt}\int_\Omega Z^p
		+\frac{c_3p}{2}\int_\Omega Z^p
		+c_7\|v(\cdot,t)\|_{L^\infty(\Omega)}^\alpha
		\|\nabla Z^{p/2}\|_{L^2(\Omega)}^2\nonumber\\
		&\qquad\le
		c_8p^2\|v(\cdot,t)\|_{L^\infty(\Omega)}^\alpha
		e^{2\sigma t}\int_\Omega Z^p+c_8^p
		\quad\mbox{for }t\ge t_0,
	\end{align*}
	where $c_8\ge1$ is independent of $p$ and $t$.
	The two-dimensional Gagliardo--Nirenberg inequality
	and Young's inequality yield
	\begin{align*}
		c_8p^2\|v(\cdot,t)\|_{L^\infty(\Omega)}^\alpha
		e^{2\sigma t}\int_\Omega Z^p
		&\le
		\frac{c_7}{2}\|v(\cdot,t)\|_{L^\infty(\Omega)}^\alpha
		\|\nabla Z^{p/2}\|_{L^2(\Omega)}^2\\
		&\quad+
		c_{9}p^4\|v(\cdot,t)\|_{L^\infty(\Omega)}^\alpha
		e^{4\sigma t}
		\left(\int_\Omega Z^{p/2}\right)^2.
	\end{align*}
Since $\sigma<\frac{a\alpha}{8b}$,
Lemma~\ref{lemma4.4} implies that
\[
\|v(\cdot,t)\|_{L^\infty(\Omega)}^\alpha
e^{4\sigma t}
\le c_{10} e^{-(\frac{a\alpha}{2b}-4\sigma)t}
\qquad\text{for all }t\ge t_0.
\]
Consequently,
	\begin{equation}\label{4.iteration}
		\frac{d}{dt}\int_\Omega Z^p
		+\frac{c_3p}{2}\int_\Omega Z^p
		\le c_{11}p^4\left(\int_\Omega Z^{p/2}\right)^2+c_8^p
		\quad\mbox{for }t\ge t_0.
	\end{equation}
For $p_k:=2^k$ for $k\ge1$ and any finite $T>t_0$, define
\[
N_k(T):=1+\sup_{t_0\le t\le T}\int_\Omega Z^{p_k}(\cdot,t).
\]
By \eqref{4.25}, $N_1(T)$ is bounded independently of $T$. 	Moreover, positivity and classical regularity at the fixed
time $t_0>0$ imply $Z(\cdot,t_0)\in L^\infty(\Omega)$. Integrating \eqref{4.iteration} with $p=p_k$, we obtain
\begin{equation}\label{4.recursion}
	N_k(T)\le c_{12}^{2^k}+c_{13}2^{4k}N_{k-1}^2(T)
	\quad\mbox{for }k\ge2,
\end{equation}
where $c_{12},c_{13}\ge1$ are independent of $k$ and $T$. 	
To absorb the first term on the right-hand side of
\eqref{4.recursion}, set
$\widehat N_k(T):=\max\{N_k(T),c_{12}^{2^k}\}$ for $k\ge1$.
Since $\widehat N_{k-1}(T)\ge c_{12}^{2^{k-1}}$,
we have
\[
c_{12}^{2^k}=\big(c_{12}^{2^{k-1}}\big)^2
\le \widehat N_{k-1}^2(T),
\]
so that \eqref{4.recursion} yields
\[
\widehat N_k(T)
\le 2c_{13}2^{4k}\widehat N_{k-1}^2(T)
\qquad\text{for all }k\ge2.
\]
Taking the $2^{-k}$-th power and setting $c_{14}:=2c_{13}$, we obtain
\[
\widehat N_k^{\frac1{2^k}}(T)
\le c_{14}^{\frac1{2^k}}2^{\frac{4k}{2^k}}
\widehat N_{k-1}^{\frac1{2^{k-1}}}(T)
\qquad\text{for all }k\ge2.
\]
Iterating this inequality gives
\[
\widehat N_k^{\frac1{2^k}}(T)
\le \widehat N_1^{\frac12}(T)\,
c_{14}^{\sum_{j=1}^k2^{-j}}\,
2^{4\sum_{j=1}^k j2^{-j}}
\le \widehat N_1^{\frac12}(T)\,c_{14}\,2^{8},
\]
where we used
$\sum_{j=1}^{\infty}2^{-j}=1$ and
$\sum_{j=1}^{\infty}j2^{-j}=2$.
Since $\widehat N_1(T)$ is bounded independently of $T$,
there exists $c_{15}>0$, independent of $k$ and $T$, such that
\[
\widehat N_k^{\frac1{2^k}}(T)\le c_{15}
\qquad\text{for all }k\ge1.
\]
Letting $k\to\infty$ and then $T\to\infty$, we conclude that
\[
\sup_{t\ge t_0}
\big(\|X_+(\cdot,t)\|_{L^\infty(\Omega)}
+\|Y_+(\cdot,t)\|_{L^\infty(\Omega)}\big)<\infty.
\]
	Finally, using $(1-s)_+\le(z-1)_+$, we conclude that
	\[
	\begin{aligned}
		\|u(\cdot,t)-u_*\|_{L^\infty(\Omega)}
		&\le u_*e^{-\sigma t}
		\left(
		\|X_+(\cdot,t)\|_{L^\infty(\Omega)}
		+\|Y_+(\cdot,t)\|_{L^\infty(\Omega)}
		\right)\\
		&\le c_{16}e^{-\sigma t}
		\quad\mbox{for }t\ge t_0.
	\end{aligned}
	\]
Together with Theorem~\ref{theorem1.1}, this yields
\eqref{4.22a} upon increasing the constant if necessary.
\end{proof}

\begin{proof}[Proof of Theorem~\ref{theorem1.2}]
The estimate \eqref{1.10} follows from Lemma~\ref{lemma4.4}.
For any admissible $\sigma$, Lemma~\ref{lem4.6} gives \eqref{1.11a}
directly.
\end{proof}
	\vspace{.2cm}
	\noindent{\bf Conflict of interest}:
	The authors declare that they have no conflicts of interest in this work.
	
	\noindent{\bf Ethics approval}:
	Ethics approval is not required for this research.
	
	\noindent{\bf Data availability statement}: Data sharing is not applicable to this article. 
	
	\noindent{\bf Acknowledgments}: This work was supported by the National Natural Science Foundation of
	China under Grant (12071030, 12171498, 12271186).


\begin{thebibliography}{99}
		
		
		
		
		
		\bibitem{Bellomo}
		N. Bellomo, A. Bellouquid, Y. S. Tao, M. Winkler, Toward a mathematical theory of Keller-Segel models of pattern formation in biological tissues, Math. Models Methods Appl. Sci., 25 (2015), 1663-1763.
		
		\bibitem{CKL}
		K. Choi, Y. Kim, M. Lee, Dweller-wanderer interaction for chemotactic phenomena, SIAM J. Appl. Math., 85(5) (2025), 2167-2189.
		
		\bibitem{Colin}
		R. Colin, B. Ni, L. Laganenka, V. Sourjik, Multiple functions of flagellar motility and chemotaxis in bacterial physiology,
		FEMS Microbiology Reviews,  45  (2021),  1--19.
		
		
		\bibitem{Fu(2012)}
		X. Fu, L. H. Tang,  C. Liu, J. D. Huang, T. Hwa, P. Lenz, Stripe formation in bacterial systems with density-suppressed motility, Phys. Rev. Lett., 108 (2012), 198102.
		
		\bibitem{Fujie2021}
		K. Fujie, J. Jiang, Comparison methods for a Keller-Segel-type model of pattern formations with density-suppressed motilities, Calc. Var. Partial Differ. Equ., 60 (2021), 92.
		
		\bibitem{Gude2020}
		S. Gude, E. Pin\c{c}e, K. M. Taute, A.-B. Seinen, T. S. Shimizu, S. J. Tans, Bacterial coexistence driven by motility and spatial competition,
		Nature, 578 (2020), 588--592.
		
		
		
		
		
		\bibitem{Jin2018}
		H. Jin, Y. Kim, Z. Wang, Boundedness, stabilization and pattern formation driven by density suppressed motility, SIAM J. Appl. Math., 78 (2018), 1632-1657.
		
		\bibitem{Jin2020}
		H. Jin, Z. Wang, Critical mass on the Keller-Segel system with signal-dependent motility, Proc. Am. Math. Soc., 148 (2020), 4855-4873.
		
		
		\bibitem{Kawasaki}
		K. Kawasaki, A. Mochizuki, M. Matsushita, T. Umeda, N. Shigesada, Modeling spatiotemporal patterns generated by bacillus subtilis, J. Theor. Biol., 188 (1997), 177-185.
		
		
		
		
		
		\bibitem{Lankeit}
		J. Lankeit, M. Winkler, Depleting the signal: Analysis of chemotaxis-consumption models--A survey, Stud. Appl. Math., 151 (2023), 1197-1229.
		
		
		
		
		
		\bibitem{Leyva(2013)}
		J. F. Leyva, C. M\'{a}laga, R. G. Plaza, The effects  of  nutrient chemotaxis  on bacterial aggregation patterns with non-linear degenerate cross diffusion, Physica A, 392 (2013), 5644-5662.
		
		
		
		\bibitem{Li Zhao2021}
		D. Li, J. Zhao, Global boundedness and large time behavior of solutions to a chemotaxis-consumption system with signal-dependent motility, Z. Angew. Math. Phys., 72 (2021), 57.
		
		\bibitem{Li Winkler2022}
		G. Li, M. Winkler, Relaxation in a Keller-Segel-consumption system involving signal-dependent motilities, Commun. Math. Sci., 21(2) (2023), 299-322.
		
		\bibitem{Li Winkler2022R}
		G. Li, M. Winkler, Refined regularity analysis for a Keller-Segel-consumption system involving signal-dependent motilities, Appl. Anal., 103 (2024), 45-64.
		
		\bibitem{LiWinkler2026}
		G. Li, M. Winkler,
		Proliferation-dominated large time behaviour in a planar doubly degenerate nutrient taxis system,
		Proc. Roy. Soc. Edinburgh Sect. A Math., (2026), 1-30.
		
		
		
		\bibitem{Liu2011}
		C. Liu, X. Fu, L. Liu, X. Ren, C. K. L. Chau, S. Li, L. Xiang, H. Zeng, G. Chen, L.-H. Tang, P. Lenz, X. Cui, W. Huang, T. Hwa, J.-D. Huang, Sequential establishment of stripe patterns in an expanding cell population, Science, 334 (2011), 238-241.
		
		
		
		
		
		
		
		\bibitem{Maini}
		P. Maini, M. Myerscough, K. Winters, J. Murray, Bifurcating spatially heterogeneous solutions in a chemotaxis model for biological pattern generation, Bull. Math. Biol., 53 (1991), 701-719.
		
		
		
		
		

\bibitem{PanMu2026}
X. Pan, C. Mu, Large time behavior in a doubly degenerate nutrient taxis
system with logistic source, Nonlinear Anal., 271 (2026), 114142.
		
		\bibitem{Tao2011}
		Y. Tao,
		Boundedness in a chemotaxis model with oxygen consumption by bacteria,
		J. Math. Anal. Appl., 381(2) (2011), 521--529.
		
		\bibitem{Tao Winkler(2012)}
		Y. Tao, M. Winkler, Eventual smoothness and stabilization of large-data solutions in a three-dimensional chemotaxis system with consumption of chemoattractant, J. Differ. Equ., 252 (2012), 2520-2543.
		
		
		\bibitem{Tao Winkler(2017)}
		Y. Tao, M. Winkler, Effects of signal-dependent motilities in a Keller-Segel-type reaction-diffusion system, Math. Models Methods Appl. Sci., 27 (2017), 1645-1683.
		
		
		
		\bibitem{Wang2025}
		L. Wang, C. Mu, Global solvability and eventual smoothness in a degenerate
		migration-consumption system with logistic source, Calc. Var. Partial Differ. Equ., 64 (2025), 264.
		
		
		
		\bibitem{Winkler(2010)}
		M. Winkler, Aggregation vs. global diffusive behavior in the higher-dimensional Keller-Segel model, J. Differ. Equ., 248 (2010), 2889-2905.
		
		
		
		
		\bibitem{WinklerDCDSB} M. Winkler, Approaching logarithmic singularities in quasilinear chemotaxis-consumption systems with signal-dependent sensitivities, Discrete Contin. Dyn. Syst. Ser. B, 27 (2022), 6565-6587.
		
		
		\bibitem{Winkler2022Moser}
		M. Winkler, Application of the Moser-Trudinger inequality in the construction of global solutions to a strongly degenerate migration model, Bull. Math. Sci., 13 (2023), 2250012.
		
		
		\bibitem{WinklerZAMP}
		M. Winkler, Global generalized solvability in a strongly degenerate taxis-type parabolic system modeling migration-consumption interaction, Z. Angew. Math. Phys., 74 (2023), 32.
		
		
		\bibitem{Winkler(2023Non)}
		M. Winkler, Stabilization despite pervasive strong cross-degeneracies in a nonlinear
		diffusion model for migration-consumption interaction, Nonlinearity, 36 (2023), 4438-4469.
		
		
		\bibitem{WinklerANS}
		M. Winkler, A degenerate migration-consumption model in domains of arbitrary dimension, Adv. Nonlinear Stud., 24 (2024), 592-615.
		
		\bibitem{Winkler2024}
		M. Winkler, A quantitative strong parabolic maximum principle and application to a taxis-type migration consumption model involving signal-dependent degenerate diffusion,
		Ann. Inst. H. Poincar\'e Anal. NonLin\'eaire, 41 (2024), 95-127.
		
		\bibitem{WinklerJDE}
		M. Winkler, $L^{\infty}$ bounds in a two-dimensional doubly degenerate nutrient taxis system with general cross-diffusive flux, J. Differ. Equ., 400 (2024), 423-456.
		
		\bibitem{WinklerYokota2018}
		M. Winkler, T. Yokota,
		Stabilization in the logarithmic Keller--Segel system,
		Nonlinear Anal., 170 (2018), 123-141.
		
		\bibitem{Xiao2023}
		Y. Xiao, J. Jiang, Global existence and uniform boundedness in a fully parabolic Keller-Segel system with non-monotonic signal-dependent motility, J. Differ. Equ., 354 (2023), 403-429.
		
		
		\bibitem{Xiao2024}
		Y. Xiao, J. Jiang, Prevention of infinite-time blowup by slightly super-linear degradation in a Keller-Segel system with density-suppressed motility, Nonlinearity, 37 (2024), 095007.
		
		
		\bibitem{Coon2017}
		C. Yoon, Y. Kim, Global existence and aggregation in a Keller-Segel model with Fokker-Planck diffusion, Acta Appl. Math., 149 (2017), 101-123.
		
		
		
	\end{thebibliography}
\end{document}